\documentclass[12pt]{amsart}
\pdfoutput=1

\usepackage[lmargin=1in,rmargin=1in,tmargin=1in,bmargin=1in]{geometry}
\usepackage{amsmath}
\usepackage{amsfonts}
\usepackage{amsthm}
\usepackage{amssymb}
\usepackage{graphicx}
\usepackage{tabularx}
\usepackage{overpic}
\usepackage{tikz}
\usetikzlibrary{arrows}
\usepackage{rotating}
\usepackage{hyperref}
\usepackage{framed}
\usepackage{xcolor}
\usepackage{placeins}
\usepackage{booktabs}

\newcommand{\x}{ {\bf x} }
\newcommand{\y}{ {\bf y} }

\newcommand{\R}{ \mathbb{R} }
\newcommand{\Q}{ \mathbb{Q} }

\newcommand{\F}{ \mathbb{F} }
\newcommand{\Z}{ \mathbb{Z} }

\newcommand{\Alg}{ \mathcal{A} }
\newcommand{\Idem}{ \mathcal{I} }

\newcommand{\Ainfty}{ \mathcal{A}_\infty }

\newcommand{\CFD}{ \widehat{\mathit{CFD}} }
\newcommand{\CFA}{ \widehat{\mathit{CFA}} }

\newcommand{\HFhat}{ \widehat{\mathit{HF}} }
\newcommand{\CFhat}{ \widehat{\mathit{CF}} }

\newcommand{\HeegDiag}{ \mathcal{H} }

\DeclareMathOperator{\id}{id}

\newtheorem{thm}{Theorem}
\newtheorem{prop}{Proposition}
\newtheorem{lem}[prop]{Lemma}
\newtheorem{question}{Question}
\newtheorem{conj}{Conjecture}

\theoremstyle{definition}
\newtheorem{definition}[prop]{Definition}

\theoremstyle{remark}
\newtheorem{remark}[prop]{Remark}

\title[Splicing Integer Framed Knot Complements]{Splicing Integer Framed Knot Complements and Bordereed {H}eegaard {F}loer Homology}
\author{Jonathan Hanselman  }

\begin{document}
\maketitle

\begin{abstract}
We consider the following question: when is the manifold obtained by gluing together two knot complements an $L$-space? Hedden and Levine proved that splicing {0-framed} complements of nontrivial knots never produces an $L$-space. We extend this result to allow for arbitrary integer framings. We find that splicing two integer framed nontrivial knot complements only produces an $L$-space if both knots are $L$-space knots and the framings lie in an appropriate range. The proof involves a careful analysis of the bordered Heegaard Floer invariants of each knot complement.
\end{abstract}

\section{Introduction}

For a rational homology 3-sphere $Y$, the rank of $\HFhat(Y)$ is bounded below by the order of $H_1(Y, \Z)$; if the rank of $\HFhat(Y)$ is equal the order of $H_1(Y, \Z)$, $Y$ is called an $L$-space. Examples of $L$-spaces include manifolds with finite fundamental group \cite{OzSz:lens_surgeries} and branched double covers of alternating links \cite{OzSz:branched_double_covers}. There is significant interest in determining exactly which 3-manifolds are $L$-spaces \cite{BoyerClay, BoyerGordonWatson, splicing, LiscaStipsicz}.

In \cite{splicing}, Hedden and Levine use a cut and paste argument to answer this question for homology spheres which are obtained by splicing together two 0-framed knot complements. Given a knot $K$ in a 3-manifold $Y$, let $X_K$ denote the manifold with boundary $Y\backslash K$ along with the curves $\mu_K$ and $\lambda_K$ in $\partial X_K$ given by the meridian and Seifert longitude of $K$, respectively. $X_K$ is the 0-framed knot complement of $K$. Given two knots $K_1 \subset Y_1$ and $K_2 \subset Y_2$, let $Y(K_1, K_2)$ denote the 3-manifold obtained by gluing $X_{K_1}$ to $X_{K_2}$ via a map $\phi:\partial X_{K_1}\to \partial X_{K_2}$taking $\mu_{K_1}$ to $\lambda_{K_2}$ and $\lambda_{K_1}$ to $\mu_{K_2}$. We refer to gluing knot complements in this way as \emph{splicing}. The main result of \cite{splicing} can be stated as follows:

\begin{thm}
\label{thm:HL_splicing}
For any homology sphere $L$-spaces $Y_1$ and $Y_2$ and any nontrivial knots $K_1 \subset Y_1$ and $K_2 \subset Y_2$, the manifold $Y(K_1, K_2)$ obtained by splicing $X_{K_1}$ and $X_{K_2}$ is not an $L$-space.
\end{thm}

The proof is based on understanding the bordered Heegaard Floer invariants of the two pieces $X_{K_1}$ and $X_{K_2}$. The existence of certain special generators in the bordered invariants implies the existence of generators in $\HFhat(Y(K_1, K_2))$. In this way, it can be shown that the rank of $\HFhat(Y(K_1, K_2))$ is at least two. The result follows using the fact that splicing 0-framed knot complements produces an integral homology sphere, so if $Y(K_1, K_2)$ is an $L$-space then $\text{rk}(\HFhat(Y)) = 1$.

\bigbreak

In this paper, we extend Theorem \ref{thm:HL_splicing} by considering splicing knot complements with non-zero framings. That is, we allow the Seifert longitude $\lambda_K$ in $\partial X_K$ to be replaced by any integer framed longitude. For a knot $K \subset Y$, let $X_K^{[n]}$ denote $Y\backslash K$, along with the curves $\mu_K$ and $\lambda_K^{[n]} = \lambda_K + n\mu_K$ in $\partial X_K$. Given two knots $K_1 \subset Y_1$ and $K_2 \subset Y_2$, define $Y(K_1^{[n_1]}, K_2^{[n_2]})$ to be the 3-manifold obtained by gluing $X_{K_1}^{[n_1]}$ to $X_{K_2}^{[n_2]}$ via a gluing map taking $\mu_{K_1}$ to $\lambda_{K_2}^{[n_2]}$ and $\lambda_{K_1}^{[n_1]}$ to $\mu_{K_2}$. The main result is the following:
\begin{thm}
\label{main_theorem}
For nontrivial knots $K_1$ and $K_2$ in $L$-space integral homology spheres, the manifold $Y(K_1^{[n_1]}, K_2^{[n_2]})$ described above in an $L$-space if and only if all of the following hold:
\begin{itemize}
\item $K_1$ and $K_2$ are $L$-space knots;
\item $n_i \ge 2 \tau(K_i)$ if $\tau(K_i) > 0$ and $n_i \le 2 \tau(K_i) $ if $\tau(K_i) < 0$;
\item if $\tau(K_1)$ and $\tau(K_2)$ have the same sign, then $n_1 \ne 2 \tau(K_1)$ or $n_2 \ne 2 \tau(K_2)$.
\end{itemize}
\end{thm}
\noindent The definition and basic properties of $L$-space knots are recalled in Section \ref{sec:Lspace_knots}. Here $\tau(K)$ denotes the Ozsv{\'a}th-Szab{\'o} concordance invariant.

The \emph{if} direction of Theorem \ref{main_theorem} can be seen by explicit tensor product computations, since the bordered Heegaard Floer invariants of an $L$-space knot complement have a well understood form; we do this in Section \ref{sec:if_direction}. The rest of Section \ref{sec:main_proof} is devoted to the proof of the \emph{only if} direction, which is broadly similar to the proof of Theorem \ref{thm:HL_splicing}. We first prove that the relevant bordered Heegaard Floer invariants contain generators satisfying certain properties. These generators, which we call \emph{durable} generators, are defined in Section \ref{sec:durable_generators}; the definition is motivated by the generators used in \cite{splicing}. Using the existence of durable generators, we can find at least two generators in $\CFhat(Y(K_1^{[n_1]}, K_2^{[n_2]}) )$ that survive in homology.

Unlike the 0-framed case, finding two generators in $\HFhat(Y(K_1^{[n_1]}, K_2^{[n_2]}) )$ is not enough to prove that $Y(K_1^{[n_1]}, K_2^{[n_2]})$ is not an $L$-space, since splicing integer framed knot complements does not, in general, produce an integral homology sphere. The key to solving this problem is the $\Z_2$ grading on (bordered) Heegaard Floer homology. By understanding the $\Z_2$ gradings of the durable generators we pick out in each bordered Heegaard Floer invariant, we can show that the two resulting generators in $\HFhat(Y(K_1^{[n_1]}, K_2^{[n_2]}))$ have different $\Z_2$ gradings. This, it turns out, is sufficient to show that $\HFhat(Y(K_1^{[n_1]}, K_2^{[n_2]}))$ is not an $L$-space.

\bigbreak

\noindent {\bf Acknowledgements: } I am  grateful to Robert Lipshitz for valuable comments on earlier drafts of this paper. I would also like to thank Adam Levine for helpful conversations and Jen Hom for answering numerous questions about knot Floer homology.

\bigbreak

\section{Background}

\subsection{Bordered Heegaard Floer homology}
Bordered Heegaard Floer homology is an invariant of 3-manifolds with parametrized boundary introduced in \cite{LOT:Bordered}. We assume the reader is familiar with the basics of bordered Heegaard Floer homology in the torus boundary case, but we review the most important definitions here.

Bordered Heegaard Floer homology associates a differential algebra to each parametrized surface. The algebra $\Alg = \Alg(T^2)$ associated to the torus is generated as a vector space over $\F = \Z_2$ by eight elements: two idempotents, $\iota_0$ and $\iota_1$, and six Reeb elements $\rho_1, \rho_2, \rho_3, \rho_{12}, \rho_{23}$, and $\rho_{123}$. The idempotents satisfy $\iota_i \iota_j = \delta_{ij} \iota_i$, and the identity element is ${\bf 1} = \iota_0 + \iota_1$. Let $\Idem$ denote the ring of idempotents. The Reeb elements interact with idempotents on either side as follows:
\[ \iota_0 \rho_1 = \rho_1 \iota_1 = \rho_1, \quad \iota_1 \rho_2 = \rho_2 \iota_0 = \rho_2, \quad \iota_0 \rho_3 = \rho_3 \iota_1 = \rho_3, \]
\[ \iota_0 \rho_{12} = \rho_{12} \iota_0 = \rho_{12}, \quad \iota_1 \rho_{23} = \rho_{23} \iota_1 = \rho_{23}, \quad \iota_0 \rho_{123} = \rho_{123} \iota_1 = \rho_{123}. \]
The only nonzero products of Reeb elements are
$\rho_1 \rho_2 = \rho_{12}$, $\rho_2 \rho_3 = \rho_{23}$, and $ \rho_1 \rho_{23} = \rho_{12} \rho_3 = \rho_{123} $. The differential on $\Alg$ is zero. For a more detailed treatment of the torus algebra see \cite[Sec 11.1]{LOT:Bordered}.

To a 3-manifold $Y$ with torus boundary and a parametrization $\phi: T^2 \to \partial Y$, we associate a right type $A$ module $\CFA(Y, \phi)$ if $\phi$ is orientation-preserving or a left type $D$-module $\CFD(Y, \phi)$ if $\phi$ is orientation-reversing (the map $\phi$ is often suppressed from the notation). These modules are invariants of the pair $(Y, \phi)$ up to homotopy equivalence. Recall that a \emph{type $A$ module} over $\Alg$ is a right $\Ainfty$-module $M$ over $\Alg$ (we can think of $\Alg$ as an $\Ainfty$-algebra with trivial higher products). Such a module has multiplication maps
\[ m_{k+1}: M \otimes_\Idem \underbrace{\Alg \otimes_\Idem \cdots \otimes_\Idem \Alg}_\text{k times} \rightarrow M \]satisfying certain $\Ainfty$ relations (see \cite[Definition 2.5]{LOT:Bordered}). A \emph{type $D$ module} over $\Alg$ is a $\Z_2$-vector space $N$ with a left action of $\Idem$ such that $N = \iota_0 N \oplus \iota_1 N$ and a map
$$\delta_1: N \to \Alg\otimes_\Idem N$$
such that 
$$(\mu\otimes \id_N)\circ(\id_\Alg \otimes \delta_1) \circ \delta_1 = 0, $$
where $\mu$ denotes multiplication on $\Alg$.

For a type $D$ module over $\Alg$, we will use the notation of coefficient maps described in \cite[Section 11.1]{LOT:Bordered}. Let $V$ be the underlying $\Z_2$-vector space of the type $D$-module. Let $\mathcal{R}$ denote the set of increasing sequences of consecutive integers in $\{1,2,3\}$ and let $\mathcal{R}' = \mathcal{R} \cup \{\emptyset\}$. Note that the set of Reeb elements in $\Alg$ is $\{\rho_I | I\in\mathcal{R}\}$. For simplicity, we define $\rho_\emptyset = {\bf 1}$. We define coefficient maps
$$D_I: V \to V$$
for each $I \in \mathcal{R}'$ such that for each $v\in V$,
$$\delta_1(v) = \sum_{I\in\mathcal{R}'} \rho_I \otimes D_I(v).$$
A type $D$ module can be represented by a directed graph: vertices correspond to generators and for generators $\x$ and $\y$ there is an arrow from the vertex $\x$ to the vertex $\y$ labelled with $D_I$ if the coefficient of $\y$ in $D_I(\x)$ is nonzero.

We say that a type $A$ module $M$ is \emph{bounded} if there is some $K$ such that for all $x \in M$, $k \ge K$ and any $I_1, \ldots, I_k \in \mathcal{R'}$, $m_{k+1}(x, \rho_{I_1}, \ldots, \rho_{I_k}) = 0$. We say that a type $D$ module $N$ is \emph{bounded} if there is some $K$ such that for all $x \in M$, $k \ge K$ and any $I_1, \ldots, I_k \in \mathcal{R'}$, $ (D_{I_k} \circ \cdots \circ D_{I_1})(y) = 0$. If either $M$ or $N$ is bounded, we can define the \emph{box tensor product} $M\boxtimes N$ to be the vector space $M \otimes_\Idem N$ equipped with the differential
$$\partial^\boxtimes(x\otimes y) = \sum_{I_1, \ldots, I_r \in \mathcal{R}}  m_{r+1}(x, \rho_{I_1}, \ldots, \rho_{I_r}) \otimes (D_{I_r} \circ \cdots \circ D_{I_1})(y).$$

\noindent Bordered Heegaard Floer invariants satisfy the following pairing theorem \cite[Theorem 1.3]{LOT:Bordered}: If $\CFA(Y_1, \phi_1)$ and $\CFD(Y_2, \phi_2)$ are bordered Heegaard Floer invariants and at least one of them is bounded, then
\begin{equation}
\CFA(Y_1, \phi_1) \boxtimes \CFD(Y_2, \phi_2) \cong \CFhat( Y_1 \cup_{\phi_2\circ\phi_1^{-1}} Y_2).
\end{equation}

Finally, recall that given a type $D$ invariant for a bordered manifold, the corresponding type $A$ invariant can be computed using an algorithm described in \cite[Section 2.3]{splicing}. There is a one-to-one correspondence between generators of $\CFD$ and generators of $\CFA$, and $\Ainfty$ operations in $\CFA$ are derived from chains of sequential coefficient maps in $\CFD$. As a convention, we will denote type $A$ generators with a bar to distinguish them from their type $D$ counterparts.

\subsection{$\Z_2$ gradings with torus boundary}

First, we review the $\Z_2$ grading in the closed case. For a closed 3-manifold $Y$, the relative $\Z_2$ grading on $\HFhat(Y)$ can be defined in terms of a genus $g$ Heegaard diagram for $Y$ with oriented $\alpha$ and $\beta$ curves. A generator $\x$ of $\HFhat(Y)$ corresponds to a $g$-tuple of intersection points $(x_1, \ldots, x_g)$, where $x_i \in \alpha_i \cap \beta_{\sigma_\x(i)}$ and $\sigma_\x$ is a permutation of $\{1, \ldots, g\}$. The permutation $\sigma_\x$ has a sign, and the orientation on the $\alpha$ and $\beta$ curves gives rise to a sign $s(x_i)$ for each intersection point $x_i$. The grading of $\x$, $gr(\x)$, is defined to be the element of $\Z_2$ such that
\[\displaystyle (-1)^{gr(\x)} = \text{sign}(\sigma_\x) \left( \prod_{i=1}^g s(x_i) \right).\]
This defines a relative $\Z_2$ grading on $\HFhat(Y)$, since it depends on the ordering of the $\alpha$ and $\beta$ curves and on their orientations. We note that the grading can be made absolute, but the relative grading is sufficient for the purposes of this paper so we will not discuss the absolute grading.

Note that the Euler characteristic of $\HFhat$ with respect to this relative grading can be interpreted as the determinant (up to sign) of the $g \times g$ matrix whose entries $M_{ij}$ are given by the signed intersection number of $\alpha_i$ and $\beta_j$. This same determinant also gives a computation of the order of $H_1(Y)$. This relationship implies the equation
\begin{equation}
\label{eq:euler_char}
\left| \text{rk}(\HFhat_1(Y)) - \text{rk}(\HFhat_0(Y)) \right| = \begin{cases} | H_1(Y) | & \text{if $Y$ is a $\Q HS$}\\ 0 & \text{otherwise}\end{cases},
\end{equation}
which leads to the inequality
$$\text{rk}(\HFhat(Y)) \ge | H_1(Y) |$$
mentioned in the introduction \cite{OzSz:properties}. The following proposition is an easy consequence of Equation \eqref{eq:euler_char}.

\begin{prop}
\label{prop:Z2_grading_Lspace}
A 3-manifold $Y$ is an $L$-space if and only if all elements of $\HFhat(Y)$ have the same $\Z_2$ grading.
\end{prop}

The relative $\Z_2$ grading was extended to bordered Heegaard Floer homology in \cite{Petkova:relative_grading}. We will only discuss the case of manifolds with torus boundary. Let $(Y, \phi: T^2\to \partial Y )$ be a bordered manifold with a genus $g$ bordered Heegaard diagram $\HeegDiag$. The bordered diagram $\HeegDiag$ contains two $\alpha$ arcs, which we label $\alpha^a_1$ and $\alpha^a_2$. The $(g-1)$ closed $\alpha$ curves are labeled $\alpha_1^c, \ldots, \alpha_{g-1}^c$, and the $\beta$ curves are labeled $\beta_1, \ldots, \beta_g$. Orient the $\alpha$ and $\beta$ curves arbitrarily and orient the $\alpha$ arcs as follows: if $(Y, \phi)$ is type $D$, label the endpoints of the $\alpha$ arcs $\alpha_1^-, \alpha_2^-, \alpha_1^+, \alpha_2^+$ starting at the basepoint and following the orientation of $-\partial\HeegDiag$ and orient the arc $\alpha_i^a$ from $\alpha_i^+$ to $\alpha_i^-$; if $(Y,\phi)$ is type $A$, label the endpoints of the $\alpha$ arcs $\alpha_1^-, \alpha_2^-, \alpha_1^+, \alpha_2^+$ starting at the basepoint and following the orientation of $\partial\HeegDiag$ and orient the arc $\alpha_i^a$ from $\alpha_i^-$ to $\alpha_i^+$ (see Figure \ref{fig:alpha_arcs}).

\begin{figure}
\begin{center}
\begin{tikzpicture}[scale = 1.3]
\node (a1) at (0,0) {};
\node (a2) at (0,1) {};
\node (a3) at (0,2) {};
\node (a4) at (0,3) {};

\node[left, color = red] at (0,0) {$\alpha_1^-$};
\node[left, color = red] at (0,1) {$\alpha_2^-$};
\node[left, color = red] at (0,2) {$\alpha_1^+$};
\node[left, color = red] at (0,3) {$\alpha_2^+$};
\node[right] at (0,.5) {$\rho_1$};
\node[right] at (0,1.5) {$\rho_2$};
\node[right] at (0,2.5) {$\rho_3$};

\draw (0,0) -- (0,1) -- (0,2) -- (0,3);
\draw[bend left = 90, min distance = 1.3cm] (0,0) to (-.7,0);
\draw[bend right = 90, min distance = 1.3cm] (0,3) to (-.7,3);
\draw (-.35, 4) -- (2, 4);
\draw (-.35, -1) -- (2, -1);
\draw (-.7,0) -- (-.7, 3);

\node at (-.35, 4) {$\bullet$};
\node[above right] at (-.35, 4) {$z$};

\draw[red] (0,0) -- (1,0);
\draw[red] (0,2) -- (1,2);
\draw[red, bend right = 45] (1,0) to node[below right = -1mm, red]{$\alpha^a_1$} (2,1);
\draw[red, bend right = 45] (2,1) to (1,2);
\draw[red] (1.9,1.1) -- (2,1) -- (2.1,1.1);

\draw[red] (0,1) -- (1,1);
\draw[red] (0,3) -- (1,3);
\draw[red, bend right = 45] (1,1) to (2,2);
\draw[red, bend right = 45] (2,2) to node[above right = -1mm, red]{$\alpha^a_2$} (1,3);
\draw[red] (1.9,2.1) -- (2,2) -- (2.1,2.1);

\end{tikzpicture} \hspace{3 cm}
\begin{tikzpicture}[scale = 1.3]
\node (a1) at (0,0) {};
\node (a2) at (0,1) {};
\node (a3) at (0,2) {};
\node (a4) at (0,3) {};

\node[left, color = red] at (0,3) {$\alpha_1^-$};
\node[left, color = red] at (0,2) {$\alpha_2^-$};
\node[left, color = red] at (0,1) {$\alpha_1^+$};
\node[left, color = red] at (0,0) {$\alpha_2^+$};
\node[right] at (0,.5) {$\rho_3$};
\node[right] at (0,1.5) {$\rho_2$};
\node[right] at (0,2.5) {$\rho_1$};

\draw (0,0) -- (0,1) -- (0,2) -- (0,3);
\draw[bend left = 90, min distance = 1.3cm] (0,0) to (-.7,0);
\draw[bend right = 90, min distance = 1.3cm] (0,3) to (-.7,3);
\draw (-.35, 4) -- (2, 4);
\draw (-.35, -1) -- (2, -1);
\draw (-.7,0) -- (-.7, 3);

\node at (-.35, 4) {$\bullet$};
\node[above right] at (-.35, 4) {$z$};

\draw[red] (0,0) -- (1,0);
\draw[red] (0,2) -- (1,2);
\draw[red, bend right = 45] (1,0) to node[below right = -1mm, red]{$\alpha^a_2$} (2,1);
\draw[red, bend right = 45] (2,1) to (1,2);
\draw[red] (1.9,1.1) -- (2,1) -- (2.1,1.1);

\draw[red] (0,1) -- (1,1);
\draw[red] (0,3) -- (1,3);
\draw[red, bend right = 45] (1,1) to (2,2);
\draw[red, bend right = 45] (2,2) to node[above right = -1mm, red]{$\alpha^a_1$} (1,3);
\draw[red] (1.9,2.1) -- (2,2) -- (2.1,2.1);

\end{tikzpicture}
\end{center}
\caption{The orientation of the $\alpha$ arcs on a bordered Heegaard diagram with type $D$ boundary (left) or type $A$ boundary (right).}
\label{fig:alpha_arcs}
\end{figure}
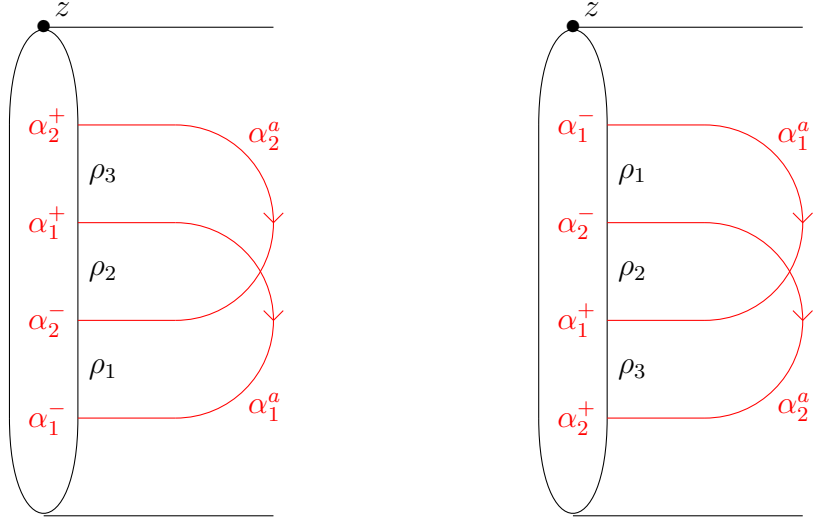

A generator of $\CFD(Y)$ or $\CFA(Y)$ corresponds to a $g$-tuple of intersection points $\x = (x_1, \ldots, x_g)$, where $x_1$ lies on $\alpha^a_1$ or $\alpha^a_2$ and $x_i$ lies on $\alpha^c_{i-1}$ for $2\le i\le g$. For each $i$, let $s(x_i)$ be the sign of the intersection of the relevant $\alpha$ arc/curve and $\beta$ curve at $x_i$. Let $\sigma_\x$ be the permutation such that $x_i$ lies on $\beta_{\sigma(i)}$ for each $i$. The $\Z_2$ grading on $\CFA(Y)$ can now be defined by
\[\displaystyle (-1)^{gr(\x)} = \text{sign}(\sigma_\x) \left( \prod_{i=1}^g s(x_i) \right).\]
The $\Z_2$ grading on $\CFD(Y)$ is defined by
\[\displaystyle (-1)^{gr(\x)} = s\left(o(\x)\right) \text{sign}(\sigma_\x) \left( \prod_{i=1}^g s(x_i) \right),\]
where $s(o(\x))$ is $+1$ if $\x$ occupies $\alpha_1^a$ and $-1$ if $\x$ occupies $\alpha_2^a$.

It is not difficult to see that the closed $\Z_2$ grading is recovered when two bordered manifolds are glued together. If $\x \in \CFA(Y_1)$ and $\y \in \CFD(Y_2)$, then the generator $\x \otimes \y$ of $\CFA(Y_1) \boxtimes \CFD(Y_2) \cong \CFhat( Y_1 \cup Y_2 )$ has $\Z_2$ grading
\[ gr(\x\otimes\y) = gr(\x) + gr(\y).\]

\begin{remark}
Just as in the closed case, the relative $\Z_2$ grading on bordered Heegaard Floer can be made into an absolute grading (see \cite{Petkova:absolute_grading}). However, this grading does not recover the absolute grading when two bordered manifolds are glued together. Consider, for example, $\CFD$ and $\CFA$ for the $0$-framed solid torus and the $(-2)$-framed solid torus. For any way of making the relative grading absolute on these four modules there is a pair whose tensor product has negative Euler characteristic with respect to the induced absolute grading, but $\HFhat$ always has nonnegative Euler characteristic.
\end{remark}

The grading on bordered Heegaard Floer homology specifies a grading on the algebra associated with the boundary. For the torus algebra $\Alg$, the grading is as follows:

\[\begin{array}{lll}
gr(\rho_1) = 0 & gr(\rho_2) = 1 & gr(\rho_{12}) = 1 \\
gr(\rho_3) = 0 & gr(\rho_{123}) = 1 & gr(\rho_{23}) = 1
\end{array}\]

The grading respects module multiplication in the sense that if $\rho_I$ is an element of $\Alg$ and $\x$ is a generator in $\CFD(Y)$, then
\begin{equation}
\label{eq:Z2_property1}
 gr(\rho_I \cdot \x) \equiv gr(\rho_I) + gr(\x)	 \quad \text{(mod $2$)}.
 \end{equation}
 If $\x$ is a generator in $\CFA(Y)$ and $\rho_{I_1}, \ldots, \rho_{I_k}$ are elements of $\Alg$, then
 \begin{equation}
 gr \left( m_{k+1}(\x, \rho_{I_1}, \ldots, \rho_{I_k}) \right) \equiv gr(\x) + gr(\rho_{I_1}) + \cdots + gr(\rho_{I_k}) + k + 1 \quad \text{(mod $2$)}.
 \end{equation}
The grading also satisfies 
\begin{equation}
\label{eq:Z2_property2}
gr\left(\partial\x\right) \equiv gr(\x) + 1 \quad \text{(mod $2$)}
\end{equation} for any generator $\x$ of $\CFD$. 

Note that if the directed graph corresponding to $\CFD(Y)$ is connected, the relative $\Z_2$ grading can be computed without reference to a Heegaard diagram. We simply choose the grading of one generator arbitrarily and determine the other gradings using Equations \eqref{eq:Z2_property1} and \eqref{eq:Z2_property2}. The grading on $\CFA(Y)$ can be obtained from the grading on $\CFD(Y)$ by flipping the grading of each generator with idempotent $\iota_0$.

\subsection{Knot Floer Homology}

Let $K$ be a knot in an $L$-space homology 3-sphere $Y$. Let $C^- = CFK^-(K, Y)$ denote the knot Floer complex of $K$ with ground field $\F = \Z_2$. Recall that $C^-$ is a chain complex over $\F[U]$ with a filtration
$$ \cdots \subset \mathcal{F}_i \subset \mathcal{F}_{i+1} \subset \cdots \subset C^- .$$
If $g(K)$ is the genus of $K$, then we have that $\mathcal{F}_{g(K)-1} \subsetneq \mathcal{F}_{g(K)} = C^-$, $\mathcal{F}_{-g(K)-1} \subset UC^-$, and $\mathcal{F}_{-g(K)} \not\subset UC^-$.

For any nonzero $x \in C^-$, the \emph{Alexander grading} of $x$ is $A(x) = \text{min}\{i | x\in\mathcal{F}_i\}$. Multiplication by $U$ decreases the Alexander grading by one. Let $C^\infty$ denote $CFK^\infty(K, Y) = C^- \otimes_{\F[U]} \F[U, U^{-1}]$; the filtration on $C^-$ extends to a filtration on $C^\infty$. We can picture $C^-$ and $C^\infty$ as living on the integer lattice in $\R^2$. If $x$ is a generator of $C^-$ over $\F[U]$, then the element $U^k x \in C^\infty$ corresponds to a point at $(-k, A(x) - k)$. We may assume that $C^-$ is reduced, meaning for any $x\in C^-$, $\partial x = U \cdot y + z$, where $A(z) < A(x)$. In terms of the lattice, this means that the differential only moves down and/or to the left. From $C^-$ and $C^\infty$ we construct two additional complexes: the \emph{vertical complex} $C^v = C^-/UC^-$ with induced differential $\partial^v$, and the \emph{horizontal complex} $C^h = \mathcal{F}_0(C^\infty)/\mathcal{F}_{-1}(C^\infty)$ with induced differential $\partial^h$.

We will need to work with special bases for $C^-$. Recall that the associated graded object of $C^-$ is the free $\F[U]$-module
$$\text{gr}(C^-) = \bigoplus_{i\in\Z} \mathcal{F}_i / \mathcal{F}_{i-1},$$
with induced multiplication by $U$. For any $x \in C^-$, let $[x]$ denote the image of $x$ in $\mathcal{F}_{A(x)} / \mathcal{F}_{A(x)-1} \subset \text{gr}(C^-)$. A basis $\{x_1, \ldots, x_n\}$ for $C^-$ over $\F[U]$ is called a \emph{filtered basis} if $\{[x_1], \ldots, [x_n]\}$ is a basis for $\text{gr}(C^-)$ over $\F[U]$. Any two filtered bases $\{x_1, \ldots, x_n\}$ and $\{x'_1, \ldots, x'_n\}$ are related by a \emph{filtered change of basis}: if $x_i = \Sigma_j a_{ij}x'_j$ and $x'_i = \Sigma_j b_{ij}x_j$ with $a_{ij}, b_{ij} \in \F[U]$, then $A(a_{ij}x'_j) \le A(x_i)$ and $A(b_{ij}x_j) \le A(x'_i)$ for all $i, j$. There are two particularly important types of filtered basis:

\begin{definition}
A \emph{vertically simplified basis} is a filtered basis $\{\xi_0, \ldots, \xi_{2n}\}$ for $C^-$ over $\F[U]$ such that for $j = 1, \ldots, n$,
$$A(\xi_{2j-1}) - A(\xi_{2j}) = h_j > 0 \quad \text{and} \quad \partial \xi_{2j-1} = \xi_{2j} \text{ (mod } UC^-),$$
while for $i = 0, 1, \ldots, n$,  $\partial \xi_{2i} = 0 \text{ (mod } UC^-)$. We say that there is a \emph{vertical arrow of length $h_j$} from $\xi_{2j-1}$ to $\xi_{2j}$.
\end{definition}

\begin{definition}
A \emph{horizontally simplified basis} is a filtered basis $\{\eta_0, \ldots, \eta_{2n}\}$ for $C^-$ over $\F[U]$ such that for $j = 1, \ldots, n$,
$$A(\eta_{2j}) - A(\eta_{2j-1}) = \ell_j > 0 \quad \text{and} \quad \partial \eta_{2j-1} = U^{\ell_j}\eta_{2j} \text{ (mod } \mathcal{F}_{A(\eta_{2j-1})-1}),$$
while for $i = 0, 1, \ldots, n$,  $A(\partial \eta_{2i}) < A(\eta_{2i})$. We say that there is a \emph{horizontal arrow of length $\ell_j$} from $\eta_{2j-1}$ to $\eta_{2j}$.
\end{definition}

$C^-$ always has a vertically simplified basis and a horizontally simplified basis \cite[Proposition 11.52]{LOT:Bordered}. Moreover, we can assume that the change of basis between these two bases is well behaved, according to the following proposition.

\begin{prop}\cite[Proposition 2.5]{splicing}
\label{prop:nice_basis}
There exists a vertically simplified basis $\{\xi_0, \ldots, \xi_{2n}\}$ and a horizontally simplified basis $\{\eta_0, \ldots, \eta_{2n}\}$ for $C^-$ such that, if
$$\xi_p = \sum_{q=0}^{2n} a_{p,q} \eta_q \quad\text{and}\quad \eta_p = \sum_{q=0}^{2n} b_{p,q} \xi_q,$$
where $a_{p,q}, b_{p,q} \in \F[U]$, then $a_{p,q} = 0$ whenever $A(\xi_p) \neq A(a_{p,q} \eta_q)$ and $b_{p,q} = 0$ whenever $A(\eta_p)\neq A(b_{p,q}\xi_q)$. In other words, each $\xi_p$ is an $\F[U]$-linear combination of the elements $\eta_q$ that are the same filtration level as $\xi_p$, and vice versa.
\end{prop}

Lipshitz, Ozsv{\'a}th, and Thurston describe a method for computing $\CFD$ of the complement of $K$ from $C^-$ (they treat the case of knots in $S^3$, but the proof carries over if $Y$ is an arbitrary $L$-space homology sphere). The statement involves the Ozsv{\'a}th-Szab{\'o} concordance invariant $\tau$, which can be defined in terms of a horizontally or vertically simplified basis by
$$\tau(K) = A(\xi_0) = -A(\eta_0).$$
We parametrize $\partial X_K^{[n]}$ such that $\alpha_1$ represents the meridian $\mu$ and $\alpha_2$ represents the framed longitude $\lambda^{[n]}$. Then according to \cite[Theorem 11.27 and Theorem A.11]{LOT:Bordered},  $\CFD(X_K^{[n]})$ is determined as follows:

\begin{thm}
Suppose that $\{\tilde\xi_0, \ldots, \tilde\xi_{2k}\}$ is a vertically simplified basis for $C^-$, $\{\tilde\eta_0, \ldots, \tilde\eta_{2k}\}$ is a horizontally simplified basis for $C^-$, and 
$$\tilde\xi_p = \sum_{q=0}^{2k} \tilde{a}_{p,q} \tilde\eta_q \quad\text{and}\quad \tilde\eta_p = \sum_{q=0}^{2k} \tilde{b}_{p,q} \tilde\xi_q,$$
where $\tilde{a}_{p,q}, \tilde{b}_{p,q} \in \F[U]$. Let $a_{p,q} = \tilde{a}_{p,q}|_{U=0}$ and $b_{p,q} = \tilde{b}_{p,q}|_{U=0}$. Then $\CFD(X_K^{[n]})$ satisfies the following:
\begin{itemize}
\item
The summand $\iota_0 \CFD(X_K^{[n]})$ has a basis $\{\xi_0, \ldots, \xi_{2k}\}$ and a basis $\{\eta_0, \ldots, \eta_{2k}\}$ such that
$$\xi_p = \sum_{q=0}^{2k} a_{p,q} \eta_q \quad\text{and}\quad \eta_p = \sum_{q=0}^{2k} b_{p,q} \xi_q,$$

\item
The summand $\iota_1 \CFD(X_K^{[n]})$ has dimension $\sum_{j=1}^k (h_j + \ell_j) + | n-2\tau(K) |$, with basis
$$ \bigcup_{j=1}^k \{\kappa^j_1, \ldots, \kappa^j_{h_j}\} \cup \bigcup_{j=1}^k \{\lambda^j_1, \ldots, \lambda^j_{\ell_j}\} \cup \{\mu_1, \ldots, \mu_{| n-2\tau(K) |} \}$$

\item
For $j = 1, \ldots, k$, there are coefficient maps
$$ \xi_{2j-1} \overset{D_1}\longrightarrow \kappa^j_1 \overset{D_{23}}\longleftarrow \cdots \overset{D_{23}}\longleftarrow \kappa^j_{h_j} \overset{D_{123}}\longleftarrow \xi_{2j}.$$
We call this sequence of generators a \emph{vertical chain} corresponding to the vertical arrow of length $h_j$ from $\tilde\xi_{2j-1}$ to $\tilde\xi_{2j}$.

\item
For $j = 1, \ldots, k$, there are coefficient maps
$$ \eta_{2j-1} \overset{D_3}\longrightarrow \lambda^j_1 \overset{D_{23}}\longrightarrow \cdots \overset{D_{23}}\longrightarrow \lambda^j_{\ell_j} \overset{D_2}\longrightarrow \eta_{2j}.$$
We call this sequence of generators a \emph{horizontal chain} corresponding to the horizontal arrow of length $\ell_j$ from $\tilde\xi_{2j-1}$ to $\tilde\xi_{2j}$.

\item
Depending on $t = n - 2\tau(K)$, there are additional coefficient maps
$$\begin{cases}
\xi_0 \overset{D_1}\longrightarrow \mu_1 \overset{D_{23}}\longleftarrow \cdots \overset{D_{23}}\longleftarrow \mu_t \overset{D_3}\longleftarrow \eta_{0} & t > 0 \\
\xi_0 \overset{D_{12}}\longrightarrow \eta_0 & t=0 \\
\xi_0 \overset{D_{123}}\longrightarrow \mu_1 \overset{D_{23}}\longrightarrow \cdots \overset{D_{23}}\longrightarrow \mu_t \overset{D_2}\longrightarrow \eta_0 & t < 0
\end{cases}$$
We call the generators in this sequence the \emph{unstable chain}.
\end{itemize}
\end{thm}

We will modify this description of $\CFD(X_K^{[n]})$ slightly to ensure that that we always work with bounded type $D$ modules. Specifically, if $K$ is not an $L$-space knot and $t \le 0$ we replace the unstable chain with
$$\begin{cases}
\xi_0 \overset{D_1}\longrightarrow \nu_1 \overset{D_\emptyset}\longleftarrow \nu_2 \overset{D_2}\longrightarrow \eta_0 & t=0 \\
\xi_0 \overset{D_{12}}\longrightarrow \nu_1 \overset{D_\emptyset}\longleftarrow \nu_2 \overset{D_3}\longrightarrow \mu_1 \overset{D_{23}}\longrightarrow \cdots \overset{D_{23}}\longrightarrow \mu_t \overset{D_2}\longrightarrow \eta_0 & t < 0
\end{cases}$$
This modification does not change the quasi-isomorphism type of $\CFD(X_K^{[n]})$. We also note that this modification does not impact any of the arguments in Section \ref{sec:main_proof}, since we will only consider generators away from the unstable chain unless $K$ is an $L$-space knot.

To see that $\CFD(X_K^{[n]})$ is bounded after modifying the unstable chain, recall that  a type $D$ module is bounded if the corresponding directed graph has no directed loops. Any loop in the graph corresponding to $\CFD(X_K^{[n]})$ is a collection of horizontal, vertical, and unstable chains. No directed loop may traverse a vertical chain, since it has arrows oriented in both directions. A directed loop could contain horizontal chains, but it must traverse all horizontal chains in the same direction. Since horizontal chains raise the Alexander grading, there can not be a directed loop consisting of only horizontal chains. Thus any loop must involve the unstable chain. For a non $L$-space knot, the above modification ensures that the unstable chain has arrows oriented in both directions, and so $\CFD(X_K^{[n]})$ has no directed loops. For an $L$-space knot, $\CFD(X_K^{[n]})$ has a special form (which will be described in Section \ref{sec:Lspace_knots}). The corresponding graph has only one loop, which contains the vertical chains and thus is not a directed loop.

\subsection{L-space knots}
\label{sec:Lspace_knots} We say that a knot $K$ in an $L$-space homology sphere $Y$ is an \emph{$L$-space knot} if some nontrivial surgery on $K$ produces an $L$-space. If $K$ is an $L$-space knot then the knot Floer homology of $K$ has a particularly simple form. It follows from \cite[Theorem 1.2]{OzSz:lens_surgeries} that there is a basis $\{\tilde{x}_0, \ldots, \tilde{x}_{2k}\}$ for $C^-$ such that
$$A(\tilde{x}_0) < \cdots < A(\tilde{x}_{2k})$$
and $A(\tilde{x}_i) = -A(\tilde{x}_{2k-i})$. Furthermore, if $K$ admits a positive $L$-space surgery, then this basis satisfies
$$\begin{cases}
\partial \tilde{x}_i = 0 & \text{if $i$ is even} \\
\partial \tilde{x}_i = \tilde{x}_{i-1} + U^{A(\tilde{x}_{i+1}) - A(\tilde{x}_i)} \tilde{x}_{i+1} & \text{if $i$ is odd}.
\end{cases}$$
If instead $K$ admits a negative $L$-space surgery, then the basis satisfies
$$\begin{cases}
\partial \tilde{x}_i = 0 & \text{if $i$ is odd} \\
\partial \tilde{x}_i = \tilde{x}_{i-1} + U^{A(\tilde{x}_{i+1}) - A(\tilde{x}_i)} \tilde{x}_{i+1} & \text{if $0<i<2k$ is even}\\
\partial \tilde{x}_0 = U^{A(\tilde{x}_{1}) - A(\tilde{x}_0)} \tilde{x}_{1} & \\
\partial \tilde{x}_{2k} = \tilde{x}_{2k-1} &
\end{cases}$$
A basis of this form gives rise to the staircase shape pictured in Figure \ref{stairstep}. It is clear that in either case the basis $\{\tilde{x}_0, \ldots, \tilde{x}_{2k}\}$ is both horizontally and vertically simplified.

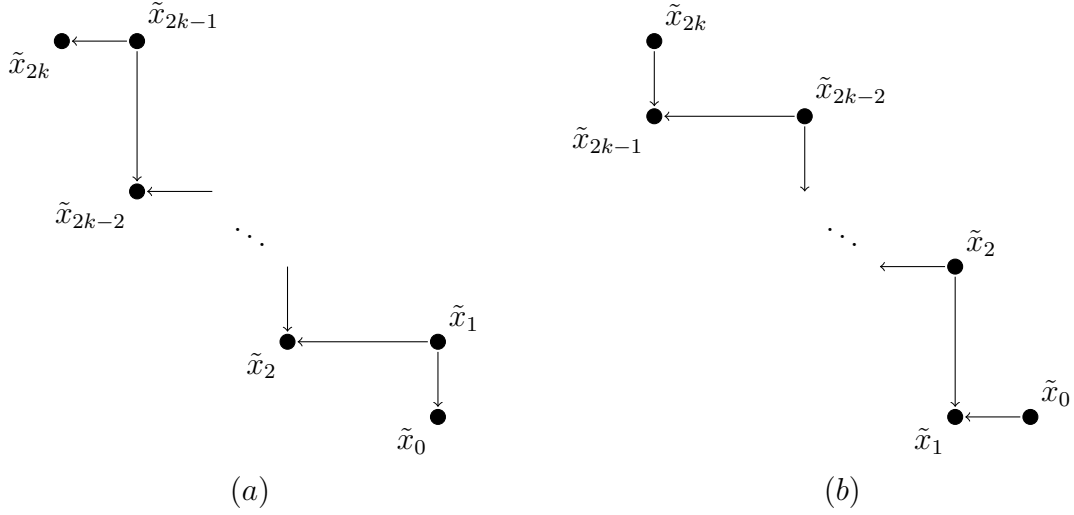
\begin{figure}
\begin{tikzpicture}
\node[circle, draw, inner sep = 2pt, outer sep = 1pt, fill = black] (x0) at (0,0) {};
\node[circle, draw, inner sep = 2pt, outer sep = 1pt, fill = black] (x1) at (0,1) {};
\node[circle, draw, inner sep = 2pt, outer sep = 1pt, fill = black] (x2) at (-2, 1) {};

\node[circle, draw, inner sep = 2pt, outer sep = 1pt, fill = black] (x4) at (-4, 3) {};
\node[circle, draw, inner sep = 2pt, outer sep = 1pt, fill = black] (x5) at (-4, 5) {};
\node[circle, draw, inner sep = 2pt, outer sep = 1pt, fill = black] (x6) at (-5,5) {};

\node[below left] at (x0) {$\tilde{x}_0$};
\node[above right] at (x1) {$\tilde{x}_1$};
\node[below left] at (x2) {$\tilde{x}_2$};
\node[below left] at (x4) {$\tilde{x}_{2k-2}$};
\node[above right] at (x5) {$\tilde{x}_{2k-1}$};
\node[below left] at (x6) {$\tilde{x}_{2k}$};

\draw[->] (x1) to (x0);
\draw[->] (x1) to (x2);
\draw[->] (-2,2) to (x2);
\draw[->] (x5) to (x4);
\draw[->] (x5) to (x6);
\draw[->] (-3,3) to (x4);
\node at (-2.5, 2.5) {$\ddots$};
\node at (-2.5, -1) {$(a)$};
\end{tikzpicture} \qquad
\begin{tikzpicture}
\node[circle, draw, inner sep = 2pt, outer sep = 1pt, fill = black] (x0) at (0,0) {};
\node[circle, draw, inner sep = 2pt, outer sep = 1pt, fill = black] (x1) at (-1,0) {};
\node[circle, draw, inner sep = 2pt, outer sep = 1pt, fill = black] (x2) at (-1, 2) {};

\node[circle, draw, inner sep = 2pt, outer sep = 1pt, fill = black] (x4) at (-3, 4) {};
\node[circle, draw, inner sep = 2pt, outer sep = 1pt, fill = black] (x5) at (-5, 4) {};
\node[circle, draw, inner sep = 2pt, outer sep = 1pt, fill = black] (x6) at (-5,5) {};

\node[above right] at (x0) {$\tilde{x}_0$};
\node[below left] at (x1) {$\tilde{x}_1$};
\node[above right] at (x2) {$\tilde{x}_2$};
\node[above right] at (x4) {$\tilde{x}_{2k-2}$};
\node[below left] at (x5) {$\tilde{x}_{2k-1}$};
\node[above right] at (x6) {$\tilde{x}_{2k}$};

\draw[<-] (x1) to (x0);
\draw[<-] (x1) to (x2);
\draw[<-] (-2,2) to (x2);
\draw[<-] (x5) to (x4);
\draw[<-] (x5) to (x6);
\draw[<-] (-3,3) to (x4);
\node at (-2.5, 2.5) {$\ddots$};
\node at (-2.5, -1) {$(b)$};
\end{tikzpicture}

\caption{A fundamental domain of $C^\infty$ for an $L$-space knot $K$ with $(a)$ $\tau(K) >0$, or $(b)$ $\tau(K) < 0$. The nodes represent the generators $\tilde{x}_0, \ldots, \tilde{x}_{2k}$ multiplied by appropriate powers of $U$, which are omitted from the diagram for simplicity. The node labelled $\tilde{x}_i$ is in fact $U^{A(\tilde{x}_i)-A(\tilde{x}_0)} \tilde{x}_i$, an element of $C^-$. }
\label{stairstep}
\end{figure}

Using this basis, it is straightforward to compute $\CFD$ for a framed complement $X_K^{[n]}$ of an $L$-space knot. $\iota_0 \CFD(X_K^{[n]})$ has basis $\{x_0, \ldots, x_{2k}\}$. For each horizontal arrow from $\tilde{x}_i$ to $\tilde{x}_{i+1}$ of length $\ell_i = A(\tilde{x}_{i+1}) - A(\tilde{x}_i)$ there is a horizontal chain
$$x_i \overset{D_3}{\longrightarrow}y^i_1 \overset{D_{23}}{\longrightarrow} \cdots \overset{D_{23}}{\longrightarrow} y^i_{\ell_i} \overset{D_2}{\longrightarrow}x_{i+1}, $$
and for each vertical arrow from $\tilde{x}_{i+1}$ to $\tilde{x}_i$ of length $\ell_i = A(\tilde{x}_{i+1}) - A(\tilde{x}_i)$ there is a vertical chain
$$x_{i+1} \overset{D_1}{\longrightarrow}y^i_1 \overset{D_{23}}{\longleftarrow} \cdots \overset{D_{23}}{\longleftarrow} y^i_{\ell_i} \overset{D_{123}}{\longleftarrow}x_i. $$
Finally, there is an unstable chain from $x_{2k}$ to $x_0$ if $\tau(K)>0$ and from $x_0$ to $x_{2k}$ if $\tau(K)<0$. Let $\ell_{2k} = | n - 2\tau(K) |$ be the length of the unstable chain. We label the generators of $\iota_1 \CFD(X_K^{[n]})$ in the unstable chain sequentially as $y^{2k}_1, \ldots, y^{2k}_{\ell_{2k}}$.

\

\section{Proof of Main Thoerem}
\label{sec:main_proof}

\subsection{Durable generators}
\label{sec:durable_generators}

Following the strategy of \cite{splicing}, we will search for special generators in $\CFD$ and $\CFA$ that give rise to generators in the homology of the box tensor product. 

\begin{definition}
\label{def:durable_generators}
Let $Y$ be a manifold with torus boundary. We call a generator $\x \in \iota_0 \CFD(Y)$ \emph{durable} if it satisfies the following conditions:
\begin{itemize}
\item $\x$ has no incoming coefficient maps; that is, $\pi_\x \circ D_I = 0$ for any $I$, where $\pi_\x$ denotes projection onto the subspace generated by $\x$.
\item If $D_{I_r} \circ \cdots \circ D_{I_1}(\x)$ is nonzero, then
\begin{itemize}
\item $I_1 = 3$ or $I_1 = 123$,
\item if $I_1 = 123$ and $r > 1$, then $I_2 = 23$,
\item if $I_1 = 3$ and $r > 1$, then $I_2 = 23$ or $I_2 = 2$,
\item if $I_2 = 2$ and $r > 2$, then $I_3 = 123$.
\end{itemize}
\end{itemize}
We call a generator $\x \in \iota_1 \CFD(Y)$ \emph{durable} if it satisfies the following:
\begin{itemize}
\item If $\pi_\x \circ D_{I_r} \circ \cdots \circ D_{I_1}$ is nonzero, then $r=1$ and $I_1 = 1$ or $I_1 = 123$.
\item If $D_{I_r} \circ \cdots \circ D_{I_1}(\x)$ is nonzero, then $I_1 = 23$.
\end{itemize}

\end{definition}

\begin{remark}
These are precisely the properties demonstrated for generators in the subspaces $B_K$ and $V_K$ in Propositions 3.5 and 3.6 of \cite{splicing}.
\end{remark}

When $\CFA$ is computed from $\CFD$ using the algorithm in \cite[Section 2.3]{splicing}, there is a direct correspondence between the generators. We define generators of $\CFA$ to be durable if they correspond to durable generators of $\CFD$. It is easy to see that this is equivalent to the following conditions (c.f. Propositions 3.7 and 3.8 in \cite{splicing}):

\begin{prop}
A durable generator $\x \in \iota_0 \CFA(Y)$ satisfies the following:
\begin{itemize}
\item There are no $\Ainfty$ operations which evaluate to $\x$, except the identity operation $m_2(\x, {1}) = \x$.
\item If $m_{r+1}(\x, a_1, \ldots, a_r)$ is nonzero for Reeb chords $a_1, \ldots, a_r$, then
	\begin{itemize}
	\item $a_1 = \rho_1, \rho_3$, or $\rho_{123}$,
	\item if $a_1 = \rho_{123}$, then $r \ge 2$ and $a_2 = \rho_2$,
	\item if $a_1 = \rho_{3}$, then $r \ge 3$, $a_2 = \rho_2$, and $a_3 = \rho_1$ or $\rho_{12}$.
	\end{itemize}
\end{itemize}
A durable generator $\x \in \iota_1 \CFA(Y)$ satisfies the following:
\begin{itemize}
\item If $m_{r+1}(\y, a_1, \ldots, a_r) = \x$ for some generator $\y \in \CFA(Y)$ and Reeb chords $a_1, \ldots, a_r$, then either $r = 1$ and $a_1 = \rho_3$ or $r = 3$ and $(a_1, a_2, a_3) = (\rho_3, \rho_2, \rho_1)$.
\item If $m_{r+1}(\x, a_1, \ldots, a_r)$ is nonzero for Reeb chords $a_1, \ldots, a_r$, then $a_1 = \rho_2$.
\end{itemize}

\end{prop}

Given these conditions, it is straightforward to check the following (c.f. \cite[Proof of Theorem 1]{splicing}):
\begin{prop}
\label{prop:tensoring_durable_generators}
If $\x$ is a durable generator of $\CFA(Y_1)$ and $\y$ is a durable generator of $\CFD(Y_2)$ such that $\x$ and $\y$ have the same idempotent, then $\x \otimes \y$ is a generator of $\CFA(Y_1) \boxtimes \CFD(Y_2)$ with no incoming or outgoing differentials. Thus, $\x \otimes \y$ survives as a generator of $\HFhat( Y_1 \cup Y_2 )$.
\end{prop}

We will also make use of a weaker condition on generators.
\begin{definition}
Let $Y$ be a manifold with torus boundary. We call a generator $\x \in \iota_0 \CFD(Y)$ \emph{weakly durable} if
$$0 = D_1(\x) = D_{12}(\x) = D_2 \circ D_{123}(\x) = D_1 \circ D_2 \circ D_3(\x) = D_{12} \circ D_2 \circ D_3(\x).$$
We call a generator $\x \in \iota_1 \CFD(Y)$ \emph{weakly durable} if $D_2(\x) = 0$ and $\pi_\x \circ D_3$ and 
$\pi_\x\circ D_1\circ D_2\circ D_3$ are trivial.
\end{definition}
The trivial chains of coefficient maps in this definition are chosen precisely to match the nontrivial $\Ainfty$ operations for a durable generator. Thus the statement in Proposition \ref{prop:tensoring_durable_generators} remains true if the generator $\y$ in $\CFD(Y_2)$ is only weakly durable.

We will find that many framed knot complements have a pair of durable generators connected by the coefficient map $D_{123}$, and that all framed knot complements have such a pair of weakly durable generators. This leads to a simple proof that certain splicings are not $L$-spaces using the following proposition.

\begin{prop}
\label{prop:durable_tensor_durable}
Let $Y_1$ and $Y_2$ be bordered 3-manifold with torus boundary. Suppose that $CFD(Y_1)$ has two durable generators $\x_1$ and $\y_1 = D_{123}(\x_1)$, and that $\CFD(Y_2)$ has two weakly durable generators $\x_2$ and $\y_2 = D_{123}(\x_2)$. Then $Y_1 \cup Y_2$ is not an $L$-space.
\end{prop}
\begin{proof}
Let $\bar\x_1$ and $\bar\y_1$ denote the generators in $\CFA(Y_1)$ corresponding to $\x_1$ and $\y_1$, respectively. $\bar\x_1 \otimes \x_2$ and $\bar\y_1 \otimes \y_2$ are generators of $\CFhat(Y_1 \cup Y_2) \cong \CFA(Y_1) \boxtimes \CFD(Y_2)$ that survive in homology. These generators have $\Z_2$ gradings
$$gr(\bar\x_1 \otimes \x_2) = gr(\bar\x_1) + gr(\x_2),$$
$$gr(\bar\y_1 \otimes \y_2) = gr(\bar\y_1) + gr(\y_2).$$
Since $D_{123}(\x_2) = \y_2$, it follows from Equations \eqref{eq:Z2_property1} and \eqref{eq:Z2_property2} that $gr(\x_2) = gr(\y_2)$. Similarly, $gr(\x_1) = gr(\y_1)$. When we compute $\CFA(Y_1)$ from $\CFD(Y_1)$, we change the grading for $\x_1$ but not for $\y_1$. As a result, $gr(\bar\x_1) \neq gr(\bar\y_1)$. This implies that $gr(\bar\x_1 \otimes \x_2) \neq gr(\bar\y_1 \otimes \y_2)$, and by Proposition \ref{prop:Z2_grading_Lspace}, $Y_1 \cup Y_2$ is not an $L$-space.
\end{proof}

\subsection{Durable generators for non-$L$-space knot}
It was shown in \cite{splicing} that for any nontrivial 0-framed knot complement, $\CFD$ has at least two durable generators. The proof relies on the form of the unstable chain and thus does not work for arbitrary framings. However, for non $L$-space knots we can use similar methods to find durable generators that do not lie on the unstable chain. Since the framing only influences the unstable chain, these durable generators exist for arbitrary framing.

Let $K$ be a nontrivial knot in an $L$-space integral homology sphere $Y$. Recall that $C^-$ will denote the knot floer complex $CFK^-(K)$. Choose simplified filtered bases $\{\tilde\xi_0, \ldots, \tilde\xi_{2m }\}$ and $\{\tilde\eta_0, \ldots, \tilde\eta_{2m} \}$ for $C^-$ as in Proposition \ref{prop:nice_basis}. For any $\tilde{a} \in C^-$, there is a corresponding element $a$ in $\iota_0 \CFD(X_K^{[n]})$. Recall that elements of $\iota_0 \CFD(X_K^{[n]})$ inherit an Alexander grading from the corresponding elements in $C^-$.

For a given $-g(K) \le k \le g(K)$, let $B_k$ denote the subspace of $\iota_0 \CFD(X_K^{[n]})$ generated by elements with Alexander grading $k$. Note that each $B_k$ has a basis which is a subset of $\{\xi_0, \ldots, \xi_{2m }\}$ and a basis which is a subset of $\{\eta_0, \ldots, \eta_{2m} \}$. Let $B'_k$ denote the subspace $B_k \cap span\{ \xi_2, \xi_4, \ldots, \xi_{2m}\} \cap span\{\eta_1, \eta_3, \ldots, \eta_{2m-1}\}$. 

\begin{lem}
\label{lem:no_left_down}
If $a \in B'_k$ for some $k$ and $D_I \circ D_2 \circ D_3(a) \ne 0$, then $I = 123$.
\end{lem}

Before approaching the general proof of Lemma \ref{lem:no_left_down}, it may be instructive to consider the proof under the simplifying assumption that the bases $\{\tilde\xi_0, \ldots, \tilde\xi_{2m }\}$ and $\{\tilde\eta_0, \ldots, \tilde\eta_{2m} \}$ of $C^-$ are the same up to permutation of the elements. The idea of the proof is the same but there is less notational complexity. Loosely speaking, we must show that if there is a length 1 horizontal arrow starting at $\tilde{a}$ in $C^-$, it is not followed by a downward vertical arrow.

\begin{remark}
It is not known whether $CFK^-(K)$ always admits a simultaneously horizontally and vertically simplified basis as in this simplifying assumption.
\end{remark}

\begin{proof}[Simplified proof of Lemma \ref{lem:no_left_down}]
Under the simplifying assumption, $B'_k$ is generated by elements of the form $\eta_{2i-1} = \xi_{2j}$, with $1\le i, j\le m$. Since coefficient maps are linear, it suffices to prove the statement when $a$ is a basis element. Assume without loss of generality that $a = \eta_1 = \xi_2$. We also assume that the length $\ell_1$ of the horizontal arrow from $\eta_1$ to $\eta_2$ is 1, since otherwise $D_2\circ D_3(a) = 0$. It follows that $D_2\circ D_3(\eta_1) = \eta_2$.

We need to show that $D_I(\eta_2) = 0$ unless $I$ is $123$. Note that $\eta_2 = \xi_j$ for some $j$. It is enough to show that $j \in \{2, 4, \ldots, 2m\}$, since $\eta_2$ has no outgoing horizontal chains, and vertical chains ending at $\xi_j$ only contribute to $D_{123}(\xi_j)$.

Consider the element $\tilde\xi_1$ of $C^-$. By the definition of vertically simplified basis, we have that
$$\partial\tilde\xi_1 = \tilde\xi_2 + U\beta = \tilde\eta_1 + U\beta$$
for some $\beta \in C^-$. Since $\tilde\eta_1 = \tilde\xi_2$ is in the kernel of the vertical differential, $\partial\tilde\eta_1 \in UC^-$. By the definition of horizontally simplified basis,
$$\partial\tilde\eta_1 = U\tilde\eta_2 + U\gamma = U\tilde\xi_j + U\gamma$$
for some $\gamma \in C^-$ with $A(\gamma) \le A(\eta_1) = k$.

Now consider
$$\partial^2 (\tilde\xi_1) = \partial(\tilde\eta_1) + \partial(U\beta) = U\tilde\xi_j + U\gamma + U\partial\beta.$$
Since multiplying by $U$ is injective, we have that $0 = \tilde\xi_j + \gamma + \partial\beta$. We consider the restriction of this equation to $U = 0$. $\gamma$ is congruent modulo $U$ to a linear combination of $\{\tilde\xi_i | A(\tilde\xi_i) \le k\}$ and $\beta$ is congruent modulo $U$ to a linear combination of $\{\tilde\xi_2, \tilde\xi_4, \ldots, \tilde\xi_{2m}\}$. Since the Alexander grading of $\tilde\xi_j = \tilde\eta_2$ is $k+1$, it follows that $j \in \{2, 4, \ldots, 2m\}$.

\end{proof}

\begin{proof}[Full proof of Lemma \ref{lem:no_left_down}]
Let $a =\sum_{i=1}^m a_i \eta_{2i-1} = \sum_{i=1}^m b_i \xi_{2i}$, where $a_i, b_i \in \F$. There is a corresponding element of $C^-$, $\tilde{a} = \sum_{i=1}^m a_i \tilde\eta_{2i-1}$; we also have that $\tilde{a}$ is congruent modulo $U$ to $\sum_{i=1}^m b_i \tilde\xi_{2i}$. For $i = 1, \ldots, m$,  define $a'_i$ to be $a_i$ if the length $\ell_i$ of the horizontal arrow from $\tilde\eta_{2i-1}$ to $\tilde\eta_{2i}$ is one and $0$ otherwise. We have that
$$D_2 \circ D_3(a) = \sum_{i=1}^m a'_i \eta_{2i} =: c.$$
We need to show that $D_1(c) = D_{12}(c) = D_3(c) = 0$. In terms of the vertical basis, we have $c = \sum_{j=0}^{2m} c_j \xi_{j}$, where $c_j \in \F$. It suffices to show that $c_j = 0$ unless $j\in\{2, 4, \ldots, 2m\}$, since $c$ has no outgoing horizontal chains and the vertical chains ending in $\xi_j$ with $j\in\{2,4,\ldots,2m\}$ only contribute outoing $D_{123}$ coefficient maps.

Consider the element $\tilde{b} = \sum_{i=1}^m b_i \tilde\xi_{2i-1}$ of $C^-$. By the definition of vertically simplified basis, $\partial\tilde{b}$ is congruent modulo $U$ to $\sum_{i=1}^m b_i \tilde\xi_{2i}$, which is congruent to $\tilde{a}$. That is,
$$\partial\tilde{b} = \tilde{a} + U\beta$$
for some $\beta \in C^-$. Since $\tilde{a}$ is congruent modulo $U$ to a linear combination of $\{\tilde\xi_2, \tilde\xi_4, \ldots, \tilde\xi_{2m}\}$, $\partial\tilde{a} \in UC^-$. By the definition of horizontally simplified basis, we have that
$$\partial\tilde{a} = U\sum_{i=1}^m a'_i \tilde\eta_{2i} + U^2 \sum_{i=1}^m (a_i - a'_i) U^{\ell_i-2} \tilde\eta_{2i} + U\gamma$$
for some $\gamma \in C^-$ with $A(\gamma) \le A(\tilde{a}) = k$. Now consider $\partial^2 \tilde{b}$:
$$0 = \partial^2(\tilde{b}) = \partial(\tilde{a}) + \partial(U\beta) =  U\sum_{i=1}^m a'_i \tilde\eta_{2i} + U^2 \sum_{i=1}^m (a_i - a'_i) U^{\ell_i-2} \tilde\eta_{2i} + U\gamma + U\partial\beta.$$
Dividing by $U$ and restricting to $U=0$, we find that
$$\sum_{i=1}^m a'_i \tilde\eta_{2i} + \gamma + \partial\beta \equiv 0 \text{ (mod } U).$$
Since  $\sum_{i=1}^m a'_i \eta_{2i} =  \sum_{j=0}^{2m} c_j \xi_{j}$, it follows that $\sum_{i=1}^m a'_i \tilde\eta_{2i}$ is congruent to $\sum_{j=0}^{2m} c_j \tilde\xi_{j}$ modulo $U$. Note that $c_j = 0$ unless $A(\tilde\xi_j) = k+1$, since $a'_i$ is only nonzero if $A(\tilde\eta_{2i}) = k+1$. Since $A(\gamma) \le k$, $\gamma$ is congruent modulo $U$ to a linear combination of $\{\tilde\xi_j | A(\tilde\xi_j) \le k\}$. Thus there can be no cancelation between the first two terms above. Finally, $\partial\beta$ is congruent modulo $U$ to a linear combination of $\{\tilde\xi_2, \tilde\xi_4, \ldots, \tilde\xi_{2m}\}$, so we must have that $c_j = 0$ unless $j \in \{2, 4, \ldots, 2m\}$.

\end{proof}

\begin{lem}
\label{lem:no_up_right}
For any $-g(K) \le k \le g(K)$ and any nonzero $a \in B'_k$, there does not exist an element $b \in \CFD(X_K^{[n]})$ such that $D_1\circ D_2(b) = D_{123}(a)$ or   $D_1\circ D_{12}(b) = D_{123}(a)$.
\end{lem}

As with the previous Lemma, we first give the simpler proof under the assumption that the bases $\{\tilde\xi_i\}$ and $\{\tilde\eta_i\}$ can be identified. We make the further simplifying assumption that $a$ is a basis element.

\begin{proof}[Simplified proof]
Under the simplifying assumption, $B'_k$ is generated by basis elements of the form $\xi_{2i} = \eta_{2j-1}$. We assume without loss of generality that $a = \eta_1 = \xi_2$. Suppose there exist $b, c \in \CFD(X_K^{[n]})$ such that $D_1(c) = D_{123}(a)$ and $c = D_2(b)$ or $c = D_{12}(b)$. We will produce a contradiction, implying that such a $b$ does not exist.

The coefficient map $D_{123}$ on $a = \xi_2$ arises from the vertical chain from $\xi_1$ to $\xi_2$. The form of the vertical chain implies that $c$ only exists if the length $h_1$ of the vertical arrow from $\xi_1$ to $\xi_2$ is one. In this case, $c$ is $\xi_1$ plus a linear combination of $\{\xi_0, \xi_2, \ldots, \xi_{2m}\}$. $\xi_1 = \eta_j$ for some $j$. In fact, $j$ must be even because the coefficient maps $D_2$ and $D_{12}$ only appear at the end of horizontal and unstable chains and thus $D_2(b)$ and $D_{12}(b)$ are linear combinations of $\{\eta_0, \eta_2, \ldots, \eta_{2m}\}$.

Consider the element $\tilde\xi_1 = \tilde\eta_j$ of $C^-$. Since $j$ is even, $\tilde\eta_j$ is in the kernel of the horizontal differential. It follows that $\partial\tilde\eta_j = \tilde\xi_2 + U\beta$ where $A(\beta) \le A(\tilde\eta_j) = k+1$. Similarly, $\partial\tilde\xi_2 = \partial \tilde\eta_1 = U^\ell \tilde\eta_2 + U\gamma$, where $A(\gamma) \le k$.  Let $\beta'$ denote the restriction of $\beta$ to $\mathcal{F}_{k+1}/\mathcal{F}_k$. $\beta'$ is congruent modulo $U$ to a linear combination of $\{\tilde\eta_0, \tilde\eta_2, \tilde\eta_3, \ldots \tilde\eta_{2m}\}$ ($\tilde\eta_1$ is not included because $A(\tilde\eta_1) = k$). Thus $\partial\beta' = \delta + U\epsilon$, where $\delta$ is a linear combination of $\{\tilde\eta_4, \tilde\eta_6, \ldots, \tilde\eta_{2m}\}$ and $A(\epsilon) \le k+1$. Now consider
$$0 = \partial^2 \tilde\eta_j = \partial(\tilde\xi_2) + \partial(U\beta) = U^\ell \tilde\eta_2 + +U \gamma+ U \partial(\beta).$$
The restriction of $\partial^2 \tilde\eta_j$ to $\mathcal{F}_{k}/\mathcal{F}_{k-1}$ gives
$$0 = \left[ \partial^2\tilde\eta_j \right] = \left[ U^\ell \tilde\eta_2 + U\gamma + U\partial\beta' \right] = \left[ U^\ell \tilde\eta_2 + U\delta \right] .$$
Since $\delta$ is a linear combination of basis elements independent from $\tilde\eta_2$, the right hand side cannot be zero. This is a contradiction, and so the element $b$ must not exist.
\end{proof}

\begin{proof}[Full proof of Lemma \ref{lem:no_up_right}]

Let $a = \sum_{i=1}^m a_{2i} \xi_{2i}$ with $a_{2i} \in \Z_2$, and suppose that $c \in \CFD(X_K^{[n]})$ such that $D_1(c) = D_{123}(a)$. Further suppose that $D_2(b) = c$ or $D_{12}(b) = c$ for some $b$. We will reach a contradiction, implying that such a $b$ does not exist.

Note that for vertical basis elements, $D_1(\xi_j) = 0$ if $j$ is even. If $j$ is odd, $D_1(\xi_j) \neq 0$, and $D_1(\xi_j) = D_{123}(\xi_{j+1})$ if and only if the length of the vertical chain from $\xi_j$ to $\xi_{j+1}$ is one. Thus in terms of the vertical basis we have $c = \sum_{j = 0}^{2m} c_j \xi_j$, where $c_j \in \Z_2$, $c_{2i-1} = a_{2i}$ for $i = 1, 2, \ldots, m$, and $a_{2i} = 0$ unless the length $h_i$ of the vertical chain from $\xi_{2i-1}$ to $\xi_{2i}$ is one. The coefficient maps $D_2$ and $D_{12}$ only appear at the end of horizontal and unstable chains, so the fact that $c = D_2(b)$ or $c = D_{12}(b)$ implies that $c = \sum_{i=0}^m b_{2i} \eta_{2i}$ for some $b_{2i}\in \Z_2$.

Consider the element $ \tilde{c} = \sum_{i = 0}^{m} b_{2i} \tilde\eta_{2i}$ of $C^-$ and note that $\tilde{c}$ is equivalent modulo $U$ to $\sum_{j = 0}^{2m} c_j \tilde\xi_j$. The definition of vertically simplified basis implies that
$$\partial\tilde{c} \equiv \sum_{i=1}^m c_{2i-1} \tilde\xi_{2i} \equiv \sum_{i=1}^m a_{2i} \tilde\xi_{2i} \quad (\text{mod } U).$$
Since $a = \sum_{i=1}^m a_{2i} \xi_{2i}$ is an element of $B'_k$, it can also be written in terms of the horizontal basis as $a = \sum_{i=1}^m d_{2i-1} \eta_{2i-1}$, where $d_{2i-1} = 0$ unless $A(\eta_{2i-1}) = k$. It follows that the last sum above is congruent modulo $U$ to $\sum_{i=1}^m d_{2i-1} \tilde\eta_{2i-1}$. The definition of horizontally simplified basis implies that $A(\partial\tilde{c}) < A(\tilde{c}) = k+1$. Putting all this information together, we have that
$$\partial\tilde{c} = \sum_{i=1}^m d_{2i-1} \tilde\eta_{2i-1} + U\beta,$$
where $A(\beta) \le k+1$.

Modulo $\mathcal{F}_k$, $\beta$ can be written as a linear combination of horizontal basis elements with Alexander grading $A(\tilde\eta_j) \ge k+1$. That is, $\beta = \sum_{j=0}^{2m} \tilde{e}_j \tilde\eta_j  + \epsilon$, where $\tilde{e}_j \in \F[U]$ is 0 unless $A(\tilde{e}_j \tilde\eta_j) = k+1$ and $A(\epsilon) \le k$. By the definition of horizontal basis, we have that
\begin{eqnarray*}
\partial \left( \sum_{i=1}^m d_{2i-1} \tilde\eta_{2i-1} \right) &=& \gamma_1 + \sum_{i=1}^m d_{2i-1} U^{\ell_i} \tilde\eta_{2i} \quad \text{and}\\
\partial \left( \sum_{i=0}^{2m} \tilde{e}_i \tilde\eta_i \right) &=& \gamma_2 +  \sum_{i=1}^m \tilde{e}_{2i-1} U^{\ell_i} \tilde\eta_{2i},
\end{eqnarray*}
where $A(\gamma_1) < k$ and $A(\gamma_2) \le k$. We will consider the restriction of $\partial^2 \tilde{c}$ to $\mathcal{F}_{k}/\mathcal{F}_{k-1} \subset \text{gr}(C^-)$. We have
\begin{eqnarray*}
 0 = \left[ \partial^2 \tilde{c} \right] &=& \left[ \gamma_1 + \sum_{i=1}^m d_{2i-1} U^{\ell_i} \tilde\eta_{2i} + U\left(\gamma_2 + \sum_{i=1}^m \tilde{e}_{2i-1} U^{\ell_i} \tilde\eta_{2i} \right) + U \partial\epsilon \right] \\
 &=& \left[ \sum_{i=1}^m d_{2i-1} U^{\ell_i} \tilde\eta_{2i} \right] + \left[ \sum_{i=1}^m \tilde{e}_{2i-1} U^{\ell_i + 1} \tilde\eta_{2i} \right] \\
 &=& \sum_{i=1}^m d_{2i-1} U^{\ell_i} [ \tilde\eta_{2i} ] + \sum_{i=1}^m U^{\ell_i + 1} [\tilde{e}_{2i-1} \tilde\eta_{2i} ].
\end{eqnarray*}
The first sum is nonzero, since $a \in B'_k$ is nonzero. However, terms from the second sum can not cancel with terms from the first, since $d_{2i-1}$ is nonzero only if $A(\eta_{2i-1}) = k$, and $\tilde{e}_{2i-1}$ is nonzero only if $A(\eta_{2i-1}) \ge k+1$. This is a contradiction, so the element $b$ must not exist.
\end{proof}

\begin{lem}
If $x$ is a nonzero generator in $B'_k$ for some $k$, then $x$ is a durable generator. Moreover, $D_{123}(x) = y$ is nonzero and is a durable generator.
\end{lem}

\begin{proof}
First we check that $x$ is durable. It is clear that there are no incoming coefficient maps, since $B'_k$ does not contain $\eta_{2i}$ for $i = 0, \ldots, m$. Outgoing coefficient maps from $B'_k$ can come either from horizontal chains starting with $D_3$, or from vertical chains starting with $D_{123}$. It follows that $D_1$ and $D_{12}$ are zero on $B'_k$.

Let $D_{I_r} \circ \cdots \circ D_{I_1}$ be a composition of coefficient maps which is nonzero on $x$. We have now that either $I_1 = 3$ or $I_1 = 123$. Consider first the case that $I_1 = 3$. The form of the horizontal chains implies that if $r > 1$, $I_2$ is either 23 or 2. We need to show that if $I_2 = 2$ and $r > 2$, then $I_3 = 123$. This last statement is proved in Lemma \ref{lem:no_left_down}. In the case that $I_1 = 123$, then the shape of vertical chains implies that if $r > 1$, $I_2$ must be $23$. This completes the proof that $x$ is durable.

Now consider $y = D_{123}(x)$. If $x = \sum_{i=1}^m a_i \xi_{2i} \neq 0$, then $y = \sum_{i=1}^m a_i \kappa^j_{h_j} \neq 0$. The restrictions on the outgoing chains from $x$ imply that if $D_I(y)$ is nonzero, then $I$ is 23. The form of vertical chains implies that if $\pi_y \circ D_I(z) = y$ then either $I = 1$ or $I=123$. Moreover, if $I=123$ then $z = x$. Since $x$ has no incoming coefficient maps, $\pi_y \circ D_{123} \circ D_I$ is trivial for any $I$. We also need that $\pi_y \circ D_1 \circ D_I$ is trivial for any $I$; this follows from Lemma \ref{lem:no_up_right} and the fact that $y \in D_{123}(\{x\})$. This proves that $y$ is durable.
\end{proof}

Any generator of $B'_k$ leads to the desired pair of durable generators. It only remains to show that such a generator must exist for some $k$.

\begin{prop}
Suppose $K$ is not an $L$-space knot; then $B'_k$ is nontrivial for some $k$.
\end{prop}
\begin{proof}

Note that $K$ is an $L$-space knot if and only if
\begin{equation}
\label{eqn:lspace_conditions}
\begin{cases}
\quad \bullet \quad \text{Each nonzero } B_k \text{ for } -g(K) \le k \le g(K) \text{ is one dimensional, }\\
\quad \bullet \quad \text{If } B_k \text{ contains } \eta_{2i-1}, \text{ then it contains one of } \{\xi_0, \xi_1, \xi_3, \ldots, \xi_{2m-1}\}, \\
\quad \bullet \quad \text{If } B_k \text{ contains } \eta_{2i}, \text{ then it contains one of } \{\xi_0, \xi_2, \xi_4, \ldots, \xi_{2m}\}, \\
\quad \bullet \quad \text{If } B_k \text{ contains } \xi_{2i-1}, \text{ then it contains one of } \{\eta_0, \eta_1, \eta_3, \ldots, \eta_{2m-1}\}, \\
\quad \bullet \quad \text{If } B_k \text{ contains } \xi_{2i}, \text{ then it contains one of } \{\eta_0, \eta_2, \eta_4, \ldots, \eta_{2m}\}.
\end{cases}
\end{equation}
Since $K$ is not an $L$-space knot, there is some integer $k$ such that $B_k$ does not satisfy \eqref{eqn:lspace_conditions}; let $k_0$ be the smallest such $k$. We will show that $B'_{k_0}$ is nontrivial.

First note that the the vertical basis for $B_{-g(K)}$ is a subset of $\{\xi_0, \xi_2, \ldots, \xi_{2m}\}$ and the horizontal basis is a subset of $\{\eta_0, \eta_1, \eta_3, \ldots, \eta_{2m-1}\}$, since $A(\xi_{2i-1})>A(\xi_{2i})$ and $A(\eta_{2i})>A(\eta_{2i-1})$ for $1 \le i \le m$. $B'_{-g(K)}$ is trivial only if $B_{-g(K)}$ is generated by either $\xi_0$ or $\eta_0$, in which case $B_{-g(K)}$ satisfies \eqref{eqn:lspace_conditions}. Thus if $k_0 = -g(K)$ we are done, and if $k_0 > -g(K)$ we can assume that either $\xi_0$ or $\eta_0$ generate the lowest Alexander grading.

Suppose that $k_0 > -g(K)$. We will assume first that $B_{-g(K)}$ is generated by $\eta_0$. It follows that $\xi_0$ is in the highest occupied Alexander grading, $g(K)$. In fact, by symmetry $B_{g(K)}$ is one dimensional and must be generated by $\xi_0$, and so $B_{g(K)}$ satisfies \eqref{eqn:lspace_conditions} and $k_0 < g(K)$. Suppose $B_{k_0}$ contains $\xi_{i_0}$ for some odd $i_0$. Then $\xi_{i_0+1}$ has Alexander grading $k_1 < k_0$. Since $B_{k_1}$ satisfies \eqref{eqn:lspace_conditions}, it is one dimensional and $\xi_{i_0+1} = \eta_{i_1}$ for $i_1$ even. If $i_1 \ne 0$, then $\eta_{i_1-1}$ has Alexander grading $k_2 < k_1$. It follows that $B_{k_2}$ is one dimensional and $\eta_{i_1-1} = \xi_{i_2}$ where $i_2$ is odd. We find that $\xi_{i_2+1} = \eta_{i_3}$ with $i_3$ even. Continuing in this way, we construct a chain of generators $\xi_{i_0}, \eta_{i_1}, \xi_{i_2}, \ldots$ of decreasing Alexander grading that only ends with $\eta_0$. Since $C^-/UC^-$ is finite dimensional, the chain must end. Similarly, if $B_{k_0}$ contains $\eta_{i_0}$ for some even $i_0 > 0$, then we can construct a chain of generators $\eta_{i_0}, \xi_{i_1} = \eta_{i_0-1}, \eta_{i_2}=\xi_{i_1+1}, \ldots$ with decreasing Alexander grading. This chain must end with $\eta_0$.

Any two such chains starting from $B_{k_0}$ must be disjoint outside $B_{k_0}$. Since each ends in $\eta_0$, there can be at most one such chain. Thus $B_{k_0}$ contains either: $(a)$ at most one of $\{\xi_1, \xi_3, \ldots, \xi_{2m-1}\}$ and none of $\{\eta_2, \eta_4, \ldots, \eta_{2m}\}$, or $(b)$ at most one of $\{\eta_2, \eta_4, \ldots, \eta_{2m}\}$ and none of $\{\xi_1, \xi_3, \ldots, \xi_{2m-1}\}$. Also note that $\eta_0$ and $\xi_0$ are not in $B_{k_0}$, since $-g(K) < k_0 < g(K)$.

If $B_{k_0}$ contains none of $\{\xi_1, \xi_3, \ldots, \xi_{2m-1}\}$ and none of $\{\eta_2, \eta_4, \ldots, \eta_{2m}\}$, then $B'_{k_0} = B_{k_0}$ is nontrivial. If $B_{k_0}$ contains $\eta_{2i}$ for some $1\le i\le m$, then $B'_{k_0} = B_{k_0}/span\{\eta_{2i}\}$. It follows that $B'_{k_0}$ is nontrivial, since if $B_{k_0} = span\{\eta_{2i}\}$ then \eqref{eqn:lspace_conditions} is satisfied. Finally, if $B_{k_0}$ contains $\xi_{2i-1}$ for some $1\le i\le m$, then $B'_{k_0} = B_{k_0}/span\{\xi_{2i-1}\}$ is nontrivial, since if $B_{k_0} = span\{\xi_{2i-1}\}$ then \eqref{eqn:lspace_conditions} is satisfied.

The case that $B_{-g(K)}$ is generated by $\eta_0$ instead of $\xi_0$ is completely identical, except that the chains of generators of decreasing Alexander grading described above terminate in $\eta_0$ instead of $\xi_0$.
\end{proof}

\subsection{Durable generators for $L$-space knots}

The pairs of durable generators described in the preceding section do not exist for $L$-space knots; indeed, for an $L$-space knot the spaces $B'_k$ are trivial for any $k$. However, we can find similar pairs of generators for certain framings. 

\begin{prop}
\label{prop:Lspace_knots_durable}
Let $K$ be an $L$-space knot with framing $n$, such that $n<2\tau(K)$ if $\tau(K)>0$ and $n>2\tau(K)+1$ if $\tau(K)<0$. Then $\CFD(X_K^{[n]})$ has a pair of durable generators $\x$ and $\y = D_{123}(x)$.
\end{prop}
\begin{proof}
Using the basis for $\CFD(X_K^{[n]})$ described in Section \ref{sec:Lspace_knots}, we simply take $\x$ to be $x_0$. $\y = D_{123}(\x)$ is $y^0_{\ell_0}$ if $\tau(K)>0$ or $y^{2k}_1$ if $\tau(K)<0$. The relevant portion of $\CFD(X_K^{[n]})$ is pictured in Figure \ref{fig:weakly_durable_generators}; it is easy to check that the generator $\x$ and $\y$ satisfy Definition \ref{def:durable_generators}.
\end{proof}

Framed complements of $L$-space knots which are not addressed by Proposition \ref{prop:Lspace_knots_durable} do not have a pair of durable generators separated by the coefficient map $D_{123}$. However, all $L$-space knot complements have a pair of weakly durable generators in $\CFD(X_K^{[n]})$. Using the basis described in Section \ref{sec:Lspace_knots}, let $\x = x_0$ and $\y = y^0_{\ell_0}$ if $\tau(K) > 0$. If $\tau(K)<0$, take $\x =x_1$ and $\y= y^1_{\ell_1}$. In either case, $D_{123}(\x) = \y$, and $\x$ and $\y$ are weakly durable. The coefficient maps into and out of $\x$ and $\y$ can be seen in Figure \ref{fig:weakly_durable_generators} if we replace the unstable chain according to the framing, as described in Section \ref{sec:Lspace_knots}.

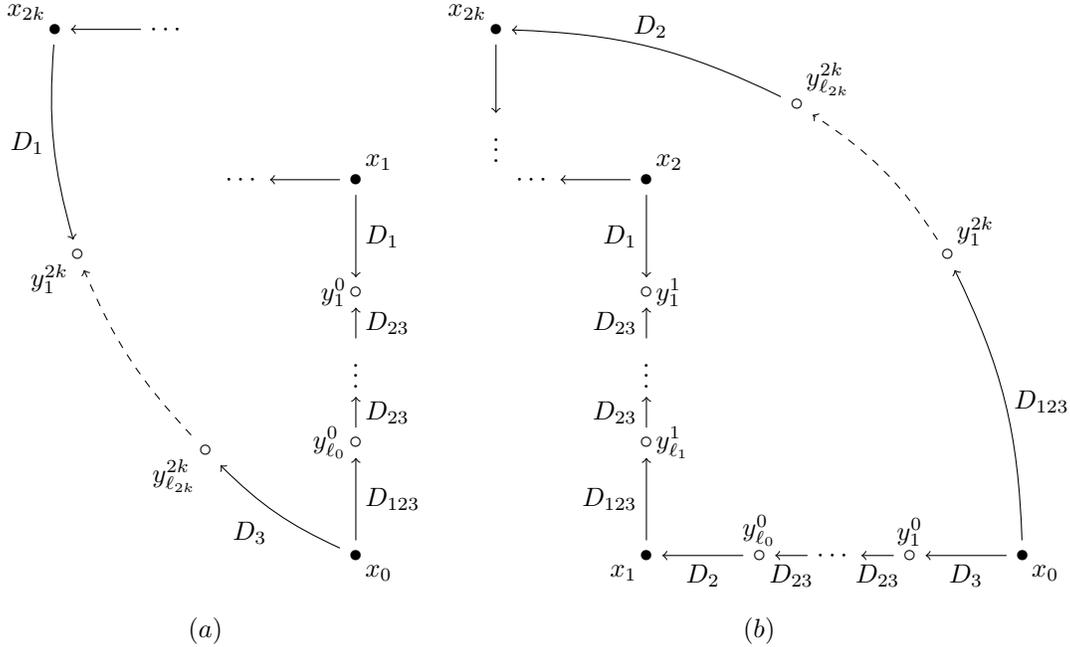
\begin{figure}

\begin{tikzpicture}
\footnotesize
\node (x0) at (0,0) {$\bullet$};
\node (y0) at (0, 1.5) {$\circ$};
\node (dots0) at (0, 2.5) {$\vdots$};
\node (y1) at (0, 3.5) {$\circ$};
\node (x1) at (0, 5) {$\bullet$};
\node (dots1) at (-1.5, 5) {$\cdots$};
\node (dots2) at (-2.5, 7) {$\cdots$};
\node (x3) at (-4, 7) {$\bullet$};

\node (z1) at (-3.7, 4) {$\circ$};
\node (z2) at (-2, 1.4) {$\circ$};
\draw[->, bend right = 10] (x3) to node[left]{$D_{1}$} (z1);
\draw[->, dashed, bend left = 10] (z2) to (z1);
\draw[->, bend left = 10] (x0) to node[below left]{$D_3$} (z2);
\node[below left] at (z1) {$y^{2k}_1$};
\node[below left] at (z2) {$y^{2k}_{\ell_{2k}}$};

\draw[->] (x0) to node[right]{$D_{123}$} (y0);
\draw[->] (y0) to node[right]{$D_{23}$} (dots0);
\draw[->] (dots0) to node[right]{$D_{23}$} (y1);
\draw[->] (x1) to node[right]{$D_1$} (y1);
\draw[->] (x1) to (dots1);
\draw[->] (dots2) to (x3);

\node[below right] at (x0) {$x_0$};
\node[left] at (y1) {$y^0_1$};
\node[left] at (y0) {$y^0_{\ell_0}$};
\node[above right] at (x1) {$x_1$};
\node[above left] at (x3) {$x_{2k}$};

\node at (-2, -1) {$(a)$};
\end{tikzpicture} \begin{tikzpicture}
\footnotesize
\node (x0) at (0,0) {$\bullet$};
\node (y0) at (-1.5, 0) {$\circ$};
\node (dots0) at (-2.5, 0) {$\cdots$};
\node (y1) at (-3.5, 0) {$\circ$};
\node (x1) at (-5, 0) {$\bullet$};
\node (y2) at (-5, 1.5) {$\circ$};
\node (dots1) at (-5, 2.5) {$\vdots$};
\node (y3) at (-5, 3.5) {$\circ$};
\node (x2) at (-5, 5) {$\bullet$};
\node (dots2) at (-6.5, 5) {$\cdots$};
\node (x3) at (-7, 7) {$\bullet$};
\node (dots3) at (-7, 5.5) {$\vdots$};

\node (z1) at (-1, 4) {$\circ$};
\node (z2) at (-3, 6) {$\circ$};
\draw[->, bend right = 12] (x0) to node[right]{$D_{123}$} (z1);
\draw[->, dashed, bend right = 12] (z1) to (z2);
\draw[->, bend right = 12] (z2) to node[above]{$D_2$} (x3);
\node[above right] at (z1) {$y^{2k}_1$};
\node[above right] at (z2) {$y^{2k}_{\ell_{2k}}$};

\draw[->] (x0) to node[below]{$D_3$} (y0);
\draw[->] (y0) to node[below]{$D_{23}$} (dots0);
\draw[->] (dots0) to node[below]{$D_{23}$} (y1);
\draw[->] (y1) to node[below]{$D_2$} (x1);
\draw[->] (x1) to node[left]{$D_{123}$} (y2);
\draw[->] (y2) to node[left]{$D_{23}$} (dots1);
\draw[->] (dots1) to node[left]{$D_{23}$} (y3);
\draw[->] (x2) to node[left]{$D_1$} (y3);
\draw[->] (x2) to (dots2);
\draw[->] (x3) to (dots3);

\node[below right] at (x0) {$x_0$};
\node[above] at (y0) {$y^0_1$};
\node[above] at (y1) {$y^0_{\ell_0}$};
\node[below left] at (x1) {$x_1$};
\node[right] at (y2) {$y^1_{\ell_1}$};
\node[right] at (y3) {$y^1_1$};
\node[above right] at (x2) {$x_2$};
\node[above left] at (x3) {$x_{2k}$};

\node at (-3.5, -1) {$(b)$};
\end{tikzpicture}
\caption{The portion of $\CFD(X_K^{[n]})$ for an $L$-space knot complement containing the pair of durable generators or the pair of weakly durable generators. $(a)$ represents a knot with $\tau(K)>0$ and $n < 2\tau(K)$; $(b)$ represents a knot with $\tau(K)<0$ and $n > 2\tau(K)$. The dotted arrow represents a chain of $D_{23}$ arrows whose length depends on $n$.}
\label{fig:weakly_durable_generators}

\end{figure}

\subsection{Proving the \emph{only if} statement}
\label{sec:boundary_case}

First note that it is sufficient to prove Theorem \ref{main_theorem} when $\tau(K_1) \ge 0$, since the result for $\tau(K_1) < 0$ follows by taking the mirror image of both framed knot complements. Using pairs of durable generators we can now prove that splicing integer framed knot complements never produces an $L$-space if at least one of the knots (we may assume it is $K_1$) is a non-$L$-space knot or has framing $n_1$ such that $n_1<2\tau(K_1)$ with $\tau(K_1)>0$. Indeed, we have shown that in this case $\CFD(X_{K_1}^{[n_1]})$ has a pair of durable generators $\x_1$ and $\y_1 = D_{123}(\x_1)$, and that $\CFD(X_{K_2}^{[n_2]})$ has a pair of weakly durable generators $\x_2$ and $\y_2 = D_{123}(\x_2)$. That the spliced manifold is not an $L$-space follows from Proposition \ref{prop:durable_tensor_durable}.

To prove the \emph{only if} direction of Theorem \ref{main_theorem}, the only case left to consider is that $K_1$ and $K_2$ are $L$-space knots, $n_1 = 2\tau(K_1)$, $n_2 = 2\tau(K_2)$, and $\tau(K_1)$ and $\tau(K_2)$ are both positive. In this case we will make use of an explicit basis for $\CFD$ of each framed complement. Let $\{x_0, \ldots, x_{2k}\}$ and $\cup_{i=0}^{2k} \{y^i_1, \ldots, y^i_{\ell_i}\}$ be the bases for $\iota_0 \CFD(X_{K_1}^{[n_1]})$ and $\iota_1 \CFD(X_{K_1}^{[n_1]})$, respectively, described in Section \ref{sec:Lspace_knots}. Let $\{u_0, \ldots, u_{2m}\}$ and $\cup_{i=0}^{2m} \{v^i_1, \ldots, v^i_{h_i}\}$ be analogous bases for $\iota_0 \CFD(X_{K_2}^{[n_2]})$ and $\iota_1 \CFD(X_{K_2}^{[n_2]})$. We use a bar to denote the corresponding type $A$ generators.

Consider the generators $\bar{x}_0 \otimes u_0$ and $\bar{y}^0_{\ell_0} \otimes v^0_{h_0}$ in $\CFA(X_{K_1}^{[n_1]}) \boxtimes \CFD(X_{K_2}^{[n_2]})$. Equations \eqref{eq:Z2_property1} and \eqref{eq:Z2_property2} imply that 
$$gr(\bar{x}_0) \neq gr(x_0) = gr(y^0_{\ell_0}) = gr(\bar{y}^0_{\ell_0})$$
and
$$gr(u_0) \neq gr(v^0_{h_0}). $$
It follows that $\bar{x}_0 \otimes u_0$ and $\bar{y}^0_{\ell_0} \otimes v^0_{h_0}$ have opposite $\Z_2$ gradings. We will show that both generators survive in homology, implying that $Y(K_1^{[n_1]}, K_2^{[n_2]})$ is not an $L$-space.

Any $\Ainfty$ operation that evaluates to $x_0$ must have $\rho_2$ as its last input. Since there is no incoming coefficient map $D_2$ at $u_0$, $x_0 \otimes u_0$ has no incoming differentials. Any nontrivial operation $m_{k+1}(x_0, \rho_{I_1}, \ldots, \rho_{I_r})$ must have $I_1 = 3$. Since $D_3(u_0) = 0$, $x_0 \otimes u_0$ has no outgoing differentials.

There are no nontrivial $\Ainfty$ operations starting at $y^0_{\ell_0}$, and if $m_{k+1}(z, \rho_{I_1}, \ldots, \rho_{I_r}) = y^0_{\ell_0}$ for some $z$ in $\CFA(X_{K_1}^{[n_1]})$ and some intervals $I_1, \ldots, I_r$, then $I_r$ is $1$, $3$, or $23$ and if $I_r=1$ then $r>1$ and $I_{r-1} = 2$. Since
$$\pi_{v^0_{h_0}} \circ D_3, \quad \pi_{v^0_{h_0}} \circ D_{23}, \quad \text{ and } \quad \pi_{v^0_{h_0}} \circ D_1 \circ D_2$$
are trivial on $\CFD(X_{K_2}^{[n_2]})$, there can be no differentials into or out of $y^0_{\ell_0} \otimes v^0_{h_0}$.

\subsection{$L$-spaces produced by splicing}
\label{sec:if_direction}

It remains to prove the \emph{if} direction of Theorem \ref{main_theorem}. That is, we need to prove that for $L$-space knots with appropriate framings the manifold $Y(K_1^{[n_1]}, K_2^{[n_2]})$ \emph{is} an $L$-space. This is more difficult in the sense that we must consider all of $\HFhat$; to show something is not an $L$-space it is sufficient to find one generator with the wrong $\Z_2$ grading, but now we must show that every generator has the same grading. Fortunately the simple form of $\CFD$ for $L$-space knot complements makes this possible.

\bigbreak

Let $K_1$ and $K_2$ be $L$-space knots and suppose that
\begin{itemize}
\item $n_i \ge 2\tau(K_i) > 0$ or $n_i \le 2\tau(K_i) < 0 $ for $i \in \{1,2\}$;
\item if $\tau(K_1)$ and $\tau(K_2)$ have the same sign, then $n_1 \neq 2\tau(K_1)$ or $n_2 \neq 2\tau(K_2)$.
\end{itemize}
Let $\{x_0, \ldots, x_{2k}\}$ and $\cup_{i=0}^{2k} \{y^i_1, \ldots, y^i_{\ell_i}\}$ be the bases for $\iota_0 \CFD(X_{K_1}^{[n_1]})$ and $\iota_1 \CFD(X_{K_1}^{[n_1]})$, respectively, described in Section \ref{sec:Lspace_knots}. Let $\{u_0, \ldots, u_{2m}\}$ and $\cup_{i=0}^{2m} \{v^i_1, \ldots, v^i_{h_i}\}$ be analogous bases for $\iota_0 \CFD(X_{K_2}^{[n_2]})$ and $\iota_1 \CFD(X_{K_2}^{[n_2]})$. We use bars to denote the corresponding type $A$ basis elements.

The $\Z_2$ grading on $\CFD(X_{K_2}^{[n_2]})$ can be computed by declaring that $gr(v^0_1) = 0$ and using Equations \eqref{eq:Z2_property1} and \eqref{eq:Z2_property2}. We find that all the generators in $\iota_1 \CFD(X_{K_2}^{[n_2]})$ have grading 0. Generators of $\iota_0 \CFD(X_{K_2}^{[n_2]})$ at the end of a horizontal or vertical chain (lower left corners) have grading 0, while those at the beginning of a horizontal or vertical chain (upper right corners) have grading 1. The computation of the $\Z_2$ grading of $\CFD(X_{K_1}^{[n_1]})$ is exactly the same, and to obtain the grading on $\CFA(X_{K_1}^{[n_1]})$ we simply switch the grading for generators with idempotent $\iota_0$.

We must prove that $Y(K_1^{[n_1]}, K_2^{[n_2]})$ is an $L$-space. Recall that
\begin{eqnarray*}
\CFhat(Y(K_1^{[n_1]}, K_2^{[n_2]}) ) &\cong& \CFA(X_{K_1}^{[n_1]}) \boxtimes \CFD(X_{K_2}^{[n_2]})\\
&\cong& \bigoplus_{\ell\in\{0,1\}} \iota_\ell \CFA(X_{K_1}^{[n_1]}) \boxtimes \iota_\ell \CFD(X_{K_2}^{[n_2]})
\end{eqnarray*}
All generators of $\iota_1 \CFA(X_{K_1}^{[n_1]})$ and $\iota_1 \CFD(X_{K_2}^{[n_2]})$ have grading 0, and thus all generators in the $\ell=1$ summand above have grading 0. We will show that all generators in the $\ell=0$ summand with grading 1 cancel in homology.

For simplicity, we assume that $\tau(K_1) > 0$ (if $\tau(K_1)<0$, the result follows by taking the mirror image of both knot complements). We consider the cases of $\tau(K_2) < 0$ and $\tau(K_2) > 0$ separately.

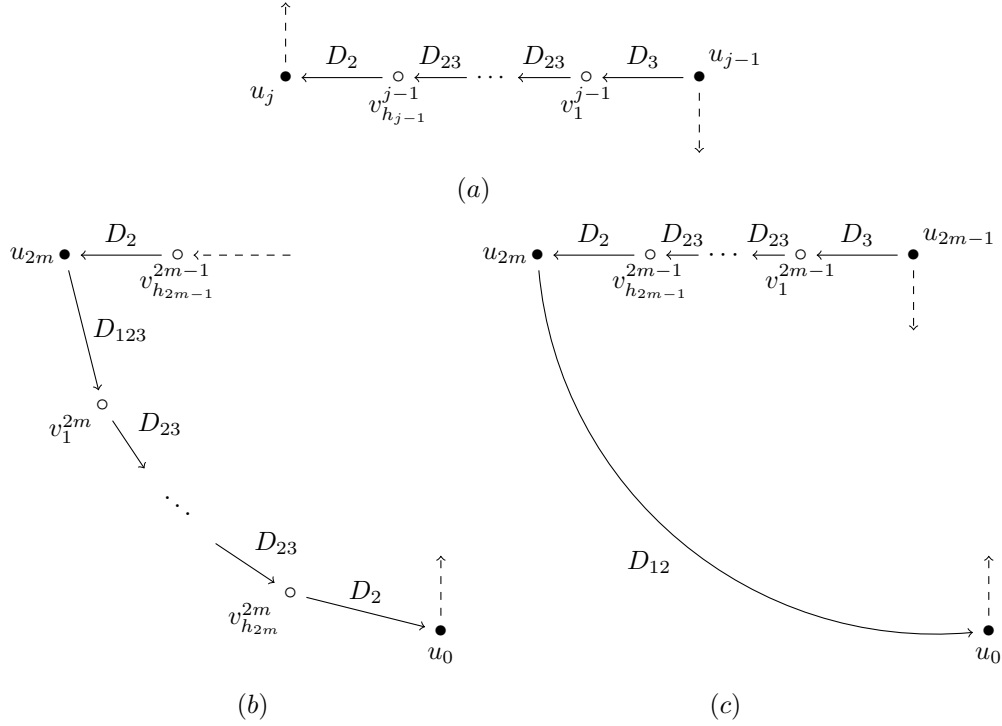
\begin{figure}
\begin{tikzpicture}
\footnotesize
\node (x2) at (0,0) {$\bullet$};
\node (y2) at (1.5,0) {$\circ$};
\node (y1) at (4,0) {$\circ$};
\node (x1) at (5.5,0) {$\bullet$};
\node (dots) at (2.75,0) {\scriptsize $\cdots$};

\draw[->] (y2) to node[above]{$D_{2}$} (x2);
\draw[->] (dots) to node[above]{$D_{23}$} (y2);
\draw[->] (y1) to node[above]{$D_{23}$} (dots);
\draw[->] (x1) to node[above]{$D_{3}$} (y1);
\draw[->, dashed] (x2) to (0,1);
\draw[->, dashed] (x1) to (5.5, -1);

\node[below left] at (x2) {$u_j$};
\node[below] at (y2) {$v^{j-1}_{h_{j-1}}$};
\node[below] at (y1) {$v^{j-1}_1$};
\node[above right] at (x1) {$u_{j-1}$};

\node at (2.5, -1.5) {$(a)$};
\end{tikzpicture}

\begin{tikzpicture}
\footnotesize
\node (x2) at (0,0) {$\bullet$};
\node (y2) at (-2,.5) {$\circ$};
\node (y1) at (-4.5,3) {$\circ$};
\node (x1) at (-5, 5) {$\bullet$};
\node (y0) at (-3.5, 5) {$\circ$};
\node (dots) at (-3.5,1.5) {$\begin{array}{c}\ddots\\ \\ \end{array}$};

\draw[->] (y2) to node[above]{$D_{2}$} (x2);
\draw[->] (dots) to node[above right]{$D_{23}$} (y2);
\draw[->] (y1) to node[above right]{$D_{23}$} (dots);
\draw[->] (x1) to node[right]{$D_{123}$} (y1);
\draw[->] (y0) to node[above]{$D_2$} (x1);
\draw[->, dashed] (x2) to (0,1);
\draw[->, dashed] (-2, 5) to (y0);

\node[below = 1 mm] at (x2) {$u_0$};
\node[below left] at (y2) {$v^{2m}_{h_{2m}}$};
\node[below left] at (y1) {$v^{2m}_1$};
\node[ left] at (x1) {$u_{2m}$};
\node[below] at (y0) {$v^{2m-1}_{h_{2m-1}}$};

\node at (-2.5, -1) {$(b)$};
\end{tikzpicture}
\begin{tikzpicture}
\footnotesize
\node (x2) at (1,0) {$\bullet$};
\node (x0) at (0,5) {$\bullet$};
\node (y1) at (-1.5,5) {$\circ$};
\node (y2) at (-3.5,5) {$\circ$};
\node (x1) at (-5, 5) {$\bullet$};
\node (dots) at (-2.5,5) {$\cdots$};

\draw[->] (y2) to node[above]{$D_{2}$} (x1);
\draw[->] (dots) to node[above]{$D_{23}$} (y2);
\draw[->] (y1) to node[above]{$D_{23}$} (dots);
\draw[->] (x0) to node[above]{$D_{3}$} (y1);
\draw[->, bend right = 45] (x1) to node[below left]{$D_{12}$} (x2);
\draw[->, dashed] (x2) to (1,1);
\draw[->, dashed] (x0) to (0,4);

\node[below = 1 mm] at (x2) {$u_0$};
\node[below] at (y2) {$v^{2m-1}_{h_{2m-1}}$};
\node[below] at (y1) {$v^{2m-1}_1$};
\node[ left] at (x1) {$u_{2m}$};
\node[above right] at (x0) {$u_{2m-1}$};

\node at (-2.5, -1) {$(c)$};
\end{tikzpicture}
\caption{The relevant portion of $\CFD(X_{K_2}^{[n_2]})$ near $u_j$ when $gr(u_j)=0$ if $(a)$ $j \neq 0$, $(b)$ $j=0$ and $n_2 > 2\tau(K_2)$, or $(c)$ $j=0$ and $n_2 = 2\tau(K_2)$.}
\label{fig:lower_left}
\end{figure}

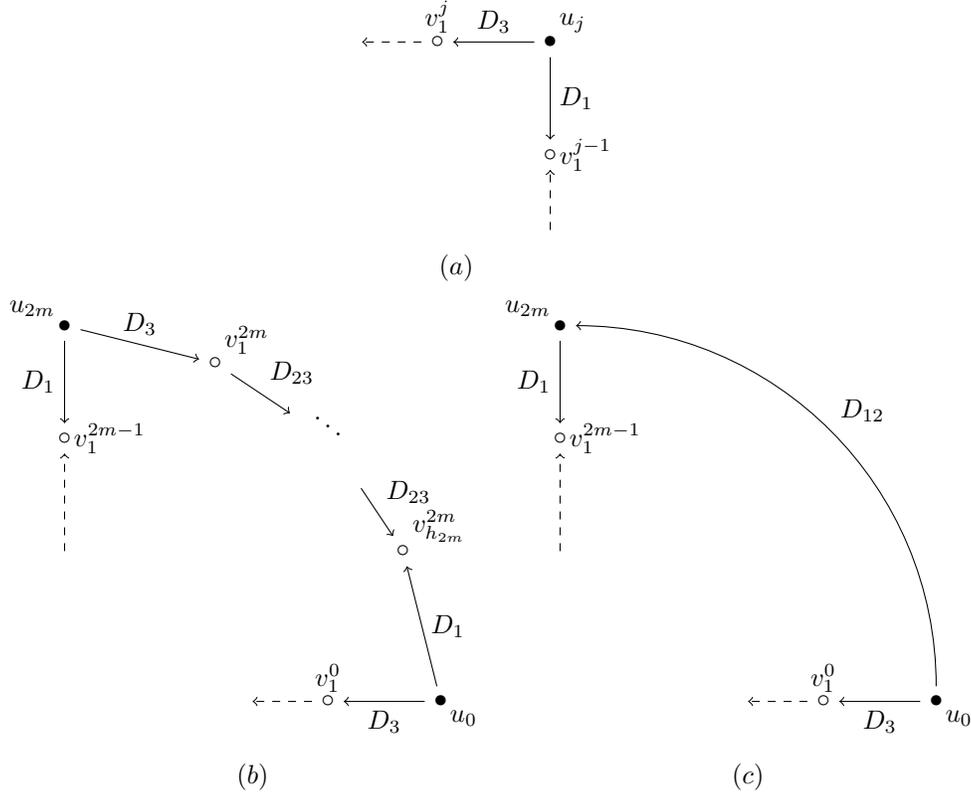
\begin{figure}
\begin{tikzpicture}
\footnotesize
\node (y2) at (0,-1.5) {$\circ$};
\node (y1) at (-1.5,0) {$\circ$};
\node (x1) at (0,0) {$\bullet$};

\draw[->] (x1) to node[above]{$D_{3}$} (y1);
\draw[->] (x1) to node[right]{$D_{1}$} (y2);
\draw[->, dashed] (y1) to (-2.5,0);
\draw[->, dashed] (0,-2.5) to (y2);

\node[right] at (y2) {$v^{j-1}_1$};
\node[above] at (y1) {$v^j_1$};
\node[above right] at (x1) {$u_j$};

\node at (-1.25, -3) {$(a)$};
\end{tikzpicture}

\begin{tikzpicture}
\footnotesize
\node (x2) at (0,0) {$\bullet$};
\node (y2) at (-.5,2) {$\circ$};
\node (y1) at (-3,4.5) {$\circ$};
\node (x1) at (-5, 5) {$\bullet$};
\node (y0) at (-5, 3.5) {$\circ$};
\node (y3) at (-1.5, 0) {$\circ$};

\node (dots) at (-1.5,3.5) {$\begin{array}{c}\ddots\\ \\ \end{array}$};

\draw[->] (x1) to node[above]{$D_{3}$} (y1);
\draw[->] (dots) to node[above right]{$D_{23}$} (y2);
\draw[->] (y1) to node[above right]{$D_{23}$} (dots);
\draw[->] (x2) to node[right]{$D_{1}$} (y2);
\draw[->] (x1) to node[left]{$D_1$} (y0);
\draw[->] (x2) to node[below]{$D_3$} (y3);

\draw[->, dashed] (y3) to (-2.5,0);
\draw[->, dashed] (-5, 2) to (y0);

\node[below right] at (x2) {$u_0$};
\node[above right] at (y2) {$v^{2m}_{h_{2m}}$};
\node[above right] at (y1) {$v^{2m}_1$};
\node[above left] at (x1) {$u_{2m}$};
\node[right] at (y0) {$v^{2m-1}_1$};
\node[above] at (y3) {$v^{0}_1$};

\node at (-2.5, -1) {$(b)$};
\end{tikzpicture}
\begin{tikzpicture}
\footnotesize
\node (x1) at (0,0) {$\bullet$};
\node (x2) at (-5, 5) {$\bullet$};
\node (y0) at (-5, 3.5) {$\circ$};
\node (y3) at (-1.5, 0) {$\circ$};

\draw[->, bend right = 45] (x1) to node[above right]{$D_{12}$} (x2);
\draw[->] (x2) to node[left]{$D_1$} (y0);
\draw[->] (x1) to node[below]{$D_3$} (y3);

\draw[->, dashed] (y3) to (-2.5,0);
\draw[->, dashed] (-5, 2) to (y0);

\node[below right] at (x1) {$u_0$};
\node[above left] at (x2) {$u_{2m}$};
\node[right] at (y0) {$v^{2m-1}_1$};
\node[above] at (y3) {$v^{0}_1$};

\node at (-2.5, -1) {$(c)$};
\end{tikzpicture}

\caption{The relevant portion of $\CFD(X_{K_2}^{[n_2]})$ near $u_j$ when $gr(u_j)=1$ if $(a)$ $j \neq 0$, $(b)$ $j=0$ and $n_2 < 2\tau(K_2)$, or $(c)$ $j=0$ and $n_2 = 2\tau(K_2)$.}
\label{fig:upper_right}
\end{figure}

\medbreak
{\bf Case 1: $\tau(K_2) < 0$. }  For $0 \le i \le 2k$, $gr(\bar{x}_i)$ is 1 if $i$ is even and 0 if $i$ is odd. For $0 \le j \le 2m$, $gr(u_j)$ is 1 if $j$ is even and 0 if $j$ is odd. So the generators in the tensor product that need to cancel in homology are $\bar{x}_i \otimes u_j$ where $i$ and $j$ have opposite parity. 

First suppose that $j$ is odd and $i$ is even. We can see in Figure \ref{fig:lower_left} that $u_j$ has an incoming $D_2$ coefficient map. More precisely, $D_2(v^{j-1}_{h_{j-1}}) = u_j$. Similarly $x_i$ has an incoming $D_2$ coefficient map unless $i=0$ and $n_1 = 2\tau(K_1)$. If $i$ is even and nonzero, then $D_2(y^{i-1}_{\ell_{i-1}}) = x_i$. According to the algorithm for computing $\CFA$ from $\CFD$, this means that $m_2(\bar{y}^{i-1}_{\ell_{i-1}}, \rho_2) = \bar{x}_i$. It is easy to check that there are no other coefficient maps into $u_j$ or $\Ainfty$ operations evaluating to $\bar{x}_i$. It follows that there is a differential from $\bar{y}^{i-1}_{\ell_{i-1}} \otimes v^{j-1}_{h_{j-1}}$ to $\bar{x}_i \otimes u_j$. If $i=0$ and $n_1 > 2\tau(K_1)$ then $D_2(y^{2k}_{\ell_{2k}}) = x_i$. It similarly follows that there is a differential from $\bar{y}^{2k}_{\ell_{2k}} \otimes v^{j-1}_{h_{j-1}}$ to $\bar{x}_i \otimes u_j$. If $i=0$ and $n_1 = 2\tau(K_1)$ then $x_i$ does not have an incoming $D_2$ coefficient map. However, in that case we have the incoming coefficient maps
$$D_{12} (x_{2k}) = x_0 \qquad \text{ and } \qquad D_{12}\circ D_2 (y^{2k-1}_{\ell_{2k-1}}) = x_0.$$
$\CFA(X_{K_1}^{[n]})$ has the corresponding $\Ainfty$ operations
$$ m_3(\bar{x}_{2k}, \rho_3, \rho_2) = \bar{x}_0 \qquad \text{ and } \qquad m_3(\bar{y}^{2k-1}_{\ell_{2k-1}}, \rho_{23}, \rho_2) = \bar{x}_0.$$
It follows that there is a differential to $\bar{x}_0 \otimes u_j$ from $\bar{x}_{2k} \otimes u_{j-1}$ if $h_{j-1} = 1$ or from $\bar{y}^{2k-1}_{\ell_{2k-1}} \otimes v^{j-1}_{h_{j-1}-1}$ if $h_{j-1} > 1$.

\begin{table}
\begin{tabular}{ccl}
\toprule
$i$, $j$ & canceling generator & \\
\midrule
$i>0$ even, $j$ odd & $\bar{y}^{i-1}_{\ell_{i-1}} \otimes v^{j-1}_{h_{j-1}}$ & \\
\midrule
 & $\bar{y}^{2k}_{\ell_{2k}} \otimes v^{j-1}_{h_{j-1}}$ & if $n_1 > 2\tau(K_1) $\\
$i=0$, $j$ odd & $\bar{x}_{2k} \otimes u_{j-1}$ & if $n_1 = 2\tau(K_1)$ and $h_{j-1} = 1$\\
 & $\bar{y}^{2k-1}_{\ell_{2k-1}} \otimes v^{j-1}_{h_{j-1}-1}$ & if $n_1 = 2\tau(K_1)$ and $h_{j-1} > 1$\\
\midrule
$i$ odd, $j>0$ even & $\bar{y}^{i}_1 \otimes v^{j-1}_1$ & \\
\midrule
$i$ odd, $j=0$ & $\bar{y}^{i-1}_1 \otimes v^{0}_1$ & \\
\bottomrule

\end{tabular}
\caption{Generators of $\CFA(X_{K_1}^{[n_1]}) \boxtimes \CFD(X_{K_2}^{[n_2]})$ which cancel in homology with $\bar{x}_i \otimes u_j$ (there is a differential to the canceling generator from $\bar{x}_i\otimes u_j$). We assume that $\tau(K_1)>0$ and $\tau(K_2)<0$. }
\label{table:canceling_generators1}

\end{table}

Now suppose that $j$ is even and $i$ is odd. We can see from Figure \ref{fig:upper_right} that $x_i$ has two outgoing coefficient maps
$$ D_1(x_i) = y^{i-1}_1 \qquad \text{ and } \qquad D_3(x_i) = y^i_1,$$
so $\bar{x}_i$ has the outgoing $\Ainfty$ operations
$$ m_2(\bar{x}_i, \rho_3) = \bar{y}^{i-1}_1 \qquad \text{ and } \qquad m_2(\bar{x}_i, \rho_1) = \bar{y}^i_1.$$
If $j = 0$ then $D_3(u_j) = v^0_1$; it follows that there is a differential from $\bar{x}_i \otimes u_j$ to $\bar{y}^{i-1}_1 \otimes v^0_1$. If $j > 0$ then $D_1(u_j) = v^{i-1}_1$ and there is a differential from $\bar{x}_i \otimes u_j$ to $\bar{y}^{i}_1 \otimes v^{i-1}_1$.

We have shown that each $\bar{x}_i\otimes u_j$ with grading 1 can be canceled with another generator in homology, as summarized in Table \ref{table:canceling_generators1}. Using the table, it is straightforward to check that all these generators can be canceled at once, that is, that none of the canceling generators are used twice. Therefore all surviving generators in $\HFhat(Y(K_1^{[n_1]}, K_2^{[n_2]}))$ have $\Z_2$ grading 0 and $Y(K_1^{[n_1]}, K_2^{[n_2]})$ is an $L$-space.

\bigbreak
{\bf Case 2: $\tau(K_2)>0$. } 
For $0 \le i \le 2k$, $gr(\bar{x}_i)$ is 1 if $i$ is even and 0 if $i$ is odd. For $0 \le j \le 2m$, $gr(u_j)$ is 0 if $j$ is even and 1 if $j$ is odd. So the generators in the tensor product that need to cancel in homology are $\bar{x}_i \otimes u_j$ where $i$ and $j$ have the same parity. 

First suppose that $i$ and $j$ are both odd. We can see from Figure \ref{fig:upper_right} that $D_1(x_i) = y^{i-1}_1$, and thus $m_2(\bar{x}_i, \rho_3) = \bar{y}^{i-1}_1$. We also see that $D_3(u_j) = v^j_1$. It follows that there is a differential in the box tensor product from $\bar{x}_i \otimes u_j$ to $\bar{y}^{i-1}_1 \otimes v^j_1$.

Now suppose that $i$ and $j$ are both even. Table \ref{table:lower_left_typeD} lists several incoming chains of coefficient maps at $u_j$, depending on $j$ and $n_2$ (see also Figure \ref{fig:lower_left}). There are similar chains of coefficient maps ending in $x_i$, and Table \ref{table:lower_left_typeA} contains the corresponding $\Ainfty$ operations which evaluate to $\bar{x}_i$.

We can find an $\Ainfty$ operation in Table \ref{table:lower_left_typeA} that pairs with a sequence of coefficient maps in Table \ref{table:lower_left_typeD} for any combination of $i$, $j$, $n_1$, and $n_2$ unless $i=j=0$, $n_1 = 2\tau(K_1)$ and $n_2 = 2\tau(K_2)$, but this case is excluded by assumption since $\tau(K_1)$ and $\tau(K_2)$ are both positive. For example, if $i>0$ and $j>0$ the operation $m_2(\bar{y}^{i-1}_{\ell_{i-1}}, \rho_2) = \bar{x}_i$ pairs with the nontrivial coefficient map $D_2(v^{j-1}_{h_{j-1}}) = u_j$ to produce a differential in the box tensor product form $\bar{y}^{i-1}_{\ell_{i-1}} \otimes v^{j-1}_{h_{j-1}}$ to $\bar{x}_i \otimes u_j$. If $i = 0$, $n_1 = 2\tau(K_1)$, $j>0$, and $h_{j-1} = 1$ then there are operations which pair in the tensor product to produce a differential from $\bar{x}_{2k} \otimes u_{j-1}$ to $\bar{x}_i \otimes u_j$.

The canceling generator for each combination is listed in Table \ref{table:canceling_generators2}. We can check that no canceling differentials are used twice, so all of these generators may be cancelled when taking homology. Since all surviving generators of $\HFhat(Y(K_1^{[n_1]}, K_2^{[n_2]}))$ have $\Z_2$ grading 0, $Y(K_1^{[n_1]}, K_2^{[n_2]})$ is an $L$-space.

\begin{center}
\begin{table}
\begin{tabular}{c}
\toprule
$j \neq 0$\\
\midrule
$\begin{array}{rcl}
D_2( v^{j-1}_{h_{j-1}} ) &=& u_j \\
D_2\circ D_3( x_{j-1}) &=& u_j \quad \text{if }h_{j-1} = 1\\
D_2\circ D_{23}( v^{j-1}_{h_{j-1}-1} ) &=& u_j \quad \text{if }h_{j-1} > 1
\end{array}$\\
\midrule
$j=0$ and $n_2 > 2\tau(K_2)$\\
\midrule
$\begin{array}{rcl}
D_2( v^{2m}_{h_{2m}} ) &=& u_j\\
D_2\circ D_{123}( u_{2m}) &=& u_j \quad \text{if }h_{2m} = 1\\
D_2\circ D_{123}\circ D_2( v^{2m-1}_{h_{2m-1}} ) &=& u_j \quad \text{if }h_{2m} = 1\\
D_2\circ D_{23}( v^{2m}_{h_{2m}-1} ) &=& u_j \quad \text{if }h_{2m} > 1
\end{array}$\\
\midrule
$j=0$ and $n_2 = 2\tau(K_2)$\\
\midrule
$\begin{array}{rcl}
D_{12}( u_{2m}) &=& u_j\\
D_{12}\circ D_2( v^{2m-1}_{h_{2m-1}} ) &=& u_j \\
D_{12}\circ D_2 \circ D_3( u_{2m-1} ) &=& u_j \quad \text{if }h_{2m-1} = 1\\
D_{12}\circ D_2\circ D_{23}( v^{2m-1}_{h_{2m-1}-1} ) &=& u_j \quad \text{if }h_{2m-1} > 1
\end{array}$\\
\midrule
\end{tabular}
\caption{Some chains of coefficient maps ending in $u_j$ for $j$ even and $\tau(K_2) > 0$.}
\label{table:lower_left_typeD}
\end{table}
\end{center}

\begin{table}
\begin{tabular}{c}
\toprule
$i \neq 0$\\
\midrule
$\begin{array}{rcl}
m_2( \bar{y}^{i-1}_{\ell_{i-1}} , \rho_2 ) &=& \bar{x}_i \\
m_2( \bar{x}_{i-1}, \rho_{12}) &=& \bar{x}_i \quad \text{if }\ell_{i-1} = 1\\
m_3( \bar{y}^{i-1}_{\ell_{i-1}-1} , \rho_2, \rho_{12}) &=& \bar{x}_i \quad \text{if }\ell_{i-1} > 1
\end{array}$\\
\midrule
$i=0$ and $n_1 > 2\tau(K_1)$\\
\midrule
$\begin{array}{rcl}
m_2( \bar{y}^{2k}_{\ell_{2k}}, \rho_2 ) &=& \bar{x}_i\\
m_4( \bar{x}_{2k}, \rho_3, \rho_2, \rho_{12}) &=& \bar{x}_i \quad \text{if }\ell_{2k} = 1\\
m_4( \bar{y}^{2k-1}_{\ell_{2k-1}}, \rho_{23}, \rho_2, \rho_{12}) &=& \bar{x}_i \quad \text{if }\ell_{2k} = 1\\
m_3( \bar{y}^{2k}_{\ell_{2k}-1}, \rho_2, \rho_{12} ) &=& \bar{x}_i \quad \text{if }\ell_{2k} > 1
\end{array}$\\
\midrule
$i=0$ and $n_1 = 2\tau(K_1)$\\
\midrule
$\begin{array}{rcl}
m_3( \bar{x}_{2k}, \rho_3, \rho_2) &=& \bar{x}_i\\
m_3( \bar{y}^{2k-1}_{\ell_{2k-1}}, \rho_{23}, \rho_2 ) &=& \bar{x}_i \\
m_3( \bar{x}_{2k-1}, \rho_{123}, \rho_2 ) &=& \bar{x}_i \quad \text{if }\ell_{2k-1} = 1\\
m_4( \bar{y}^{2k-1}_{\ell_{2k-1}-1}, \rho_2, \rho_{123}, \rho_2 ) &=& \bar{x}_i \quad \text{if }\ell_{2k-1} > 1
\end{array}$\\
\midrule
\end{tabular}
\caption{Some $\Ainfty$ operations evaluating to $\bar{x}_i$ for $i$ even and $\tau(K_1)>0$.}
\label{table:lower_left_typeA}
\end{table}

\begin{table}
\begin{tabular}{ccl}

$i$, $j$ & canceling generator & \\
\toprule
$i$ odd, $j$ odd & $\bar{y}^{i-1}_1 \otimes v^j_1$ & \\
\midrule
$i>0$ even, $j>0$ even & $\bar{y}^{i-1}_{\ell_{i-1}} \otimes v^{j-1}_{h_{j-1}}$ & \\
\midrule
 			& $\bar{y}^{i-1}_{\ell_{i-1}} \otimes v^{2m}_{h_{2m}}$ & if $n_2 > 2\tau(K_2) $\\
$i>0$ even, $j=0$ & $\bar{x}_{i-1} \otimes u_{2m}$ & if $n_2 = 2\tau(K_2)$ and $\ell_{i-1} = 1$\\
 			& $\bar{y}^{i-1}_{\ell_{i-1}-1} \otimes v^{2m-1}_{h_{2m-1}}$ & if $n_2 = 2\tau(K_2)$ and $\ell_{i-1} > 1$\\
\midrule
 			& $\bar{y}^{2k}_{\ell_{2k}} \otimes v^{j-1}_{h_{j-1}}$ & if $n_1 > 2\tau(K_1) $\\
$i=0$, $j>0$ even & $\bar{x}_{2k} \otimes u_{j-1}$ & if $n_1 = 2\tau(K_1)$ and $h_{j-1} = 1$\\
 			& $\bar{y}^{2k-1}_{\ell_{2k-1}} \otimes v^{j-1}_{h_{j-1}-1}$ & if $n_1 = 2\tau(K_1)$ and $h_{j-1} > 1$\\
\midrule
 			& $\bar{y}^{2k}_{\ell_{2k}} \otimes v^{2m}_{h_{2m}}$ & if $n_1 > 2\tau(K_1) $ and $n_2 > 2\tau(K_2)$\\
		 	& $\bar{y}^{2k-1}_{\ell_{2k-1}} \otimes v^{2m}_{h_{2m}-1}$ & if $n_1 = 2\tau(K_1)$ and $h_{2m} > 1$\\
 			& $\bar{x}_{2k-1} \otimes u_{2m}$ & if $n_1 = 2\tau(K_1)$, $h_{2m} = 1$, and $\ell_{2k-1} = 1$\\
$i=0$, $j=0$	& $\bar{y}^{2k-1}_{\ell_{2k-1}-1} \otimes v^{2m-1}_{h_{2m-1}}$ & if $n_1 = 2\tau(K_1)$, $h_{2m} = 1$, and $\ell_{2k-1} > 1$\\
			
			& $\bar{y}^{2k}_{\ell_{2k}-1} \otimes v^{2m-1}_{h_{2m-1}}$ & if $n_2 = 2\tau(K_2)$ and $\ell_{2k} > 1$\\
 			& $\bar{x}_{2k} \otimes u_{2m-1}$ & if $n_2 = 2\tau(K_2)$, $\ell_{2k} = 1$, and $h_{2m-1} = 1$\\
			& $\bar{y}^{2k-1}_{\ell_{2k-1}} \otimes v^{2m-1}_{h_{2m-1}-1}$ & if $n_2 = 2\tau(K_2)$, $\ell_{2k} = 1$, and $h_{2m-1} > 1$\\
\midrule

\end{tabular}
\caption{Generators of $\CFA(X_{K_1}^{[n_1]}) \boxtimes \CFD(X_{K_2}^{[n_2]})$ which cancel in homology with $\bar{x}_i \otimes u_j$ (there is a differential from the canceling generator to $\bar{x}_i\otimes u_j$). We assume that $\tau(K_1)>0$ and $\tau(K_2)>0$.}
\label{table:canceling_generators2}

\end{table}

\section{Future directions}

Having addressed splicing of integer framed knot complements, it is natural to ask if Theorem \ref{main_theorem} can be extended to include rational framings. Equivalently, we ask the following:
\begin{question}
\label{question}
When is a manifold produced by gluing together two knot complements using any gluing map an $L$-space?
\end{question}
The challenge in extending the proof of Theorem \ref{main_theorem} to answer Question \ref{question} is the complexity of $\CFD$ of the knot complements. For integer framings, we can easily produce a bordered invariant from $CFK^-$ and the impact of changing the framing is minimal, but the case of rational framing is less well understood. The techniques used in this paper may be valuable in answering Question \ref{question}, but we would first need a sufficiently simple description of $\CFD$ of a rationally framed knot complement.

In the meantime, we can guess an answer to Question 1 by viewing the problem in a broader context. The following conjecture is  motivated by recent work of Boyer and Clay concerning graph manifolds \cite{BoyerClay}. An important ingredient is the twisted $I$-bundle over the torus, denoted $N_2$. For a manifold $M$ with torus boundary, an $N_2$-filling of $M$ along a curve $\gamma$ in $\partial M$ will mean a manifold obtained by gluing $N_2$ to $M$ so that the rational longitude of $N_2$ is identified with $\gamma$. The results in \cite{BoyerClay} conjecturally imply that gluing together two graph manifolds along their common torus boundary produces an $L$-space if and only if there is some rational curve $\gamma$ on the boundary torus such that $N_2$-filling either manifold along $\gamma$ produces an $L$-space.

We can speculate that this principle extends beyond graph manifolds, and perhaps that it applies to gluing knot complements. This idea motivates the following conjecture:

\begin{conj}
\label{conjecture}
For $i \in \{1,2\}$, let $K_i$ be a nontrivial knot in an $L$-space homology sphere $Y_i$ with meridian $\mu_i$ and Seifert longitude $\lambda_i$. If $\tau(K_1) > 0$ let $t = 2\tau(K_1)-1$ and if $\tau(K_1)<0$ let $t = 2\tau(K_1)+1$. Let $Y$ be the manifold obtained by gluing the exterior of $K_1$ to the exterior of $K_2$ such that
\begin{eqnarray*}
\mu_1 & \text{ is identified with } & p \mu_2 + q \lambda_2 \quad \text{and} \\
\lambda_1 + t \mu_1& \text{ is identified with } &  r \mu_2 + s \lambda_2
\end{eqnarray*}
Then $Y$ is an $L$-space if and only if all of the following hold:
\begin{itemize}
\item $K_1$ and $K_2$ are $L$-space knots;
\item If $\tau(K_1)>0$ then $ \frac{p}{q}> \frac{r}{s}$; if $\tau(K_1)<0$ then $\frac{p}{q} < \frac{r}{s}$;
\item If $\tau(K_2)>0$ then $\frac{p}{q}, \frac{r}{s} \in  (2\tau(K_2)-1, \infty)$; if $\tau(K_2)<0$ then $\frac{p}{q}, \frac{r}{s} \in (-\infty, 2\tau(K_2)+1)$.
\end{itemize}
\end{conj}

We conclude by noting that Theorem \ref{main_theorem} is consistent with this conjecture. When we splice $X_{K_1}^{[n_2]}$ with $X_{K_2}^{[n_2]}$, we have the following identifications:
\begin{eqnarray*}
\mu_1 & \longleftrightarrow & n_2 \mu_2 + \lambda_2 \\
\lambda_1 + n_1 \mu_1& \longleftrightarrow &  \mu_2
\end{eqnarray*}
Adding $(t-n_1)$ copies of the first line to the second tells us that
$$\lambda_1 + t \mu_1  \longleftrightarrow   ((t-n_1)n_2 + 1)\mu_2  +  (t-n_1) \lambda_2.$$
In the notation of Conjecture \ref{conjecture}, we have
$$ \frac{p}{q} = n_2 \quad\text{ and }\quad \frac{r}{s} = n_2 + \frac{1}{t-n_1} .$$
The conditions on $p$, $q$, $r$, and $s$ in Conjecture \ref{conjecture} imply the conditions on $n_1$ and $n_2$ in Theorem \ref{main_theorem}.

\bibliography{HFbibliography}

\providecommand{\bysame}{\leavevmode\hbox to3em{\hrulefill}\thinspace}
\providecommand{\MR}{\relax\ifhmode\unskip\space\fi MR }
\providecommand{\MRhref}[2]{%
  \href{http://www.ams.org/mathscinet-getitem?mr=#1}{#2}
}
\providecommand{\href}[2]{#2}
\begin{thebibliography}{10}

\bibitem{BoyerClay}
Steven Boyer and Adam Clay, \emph{Foliations, orders, representations,
  {L}-spaces and graph manifolds}, 2014.

\bibitem{BoyerGordonWatson}
Steven Boyer, Cameron~McA. Gordon, and Liam Watson, \emph{On {L}-spaces and
  left-orderable fundamental groups}, Math. Ann. \textbf{356} (2013), no.~4,
  1213--1245. \MR{3072799}

\bibitem{splicing}
Matthew Hedden and Adam~Simon Levine, \emph{Splicing knot complements and
  bordered {F}loer homology}, 2012.

\bibitem{LOT:Bordered}
Robert Lipshitz, Peter~S. Ozsv{\'a}th, and Dylan~P. Thurston, \emph{Bordered
  {H}eegaard {F}loer homology: {I}nvariance and pairing}, 2008.

\bibitem{LiscaStipsicz}
Paolo Lisca and Andr{\'a}s~I. Stipsicz, \emph{On the existence of tight contact
  structures on {S}eifert fibered 3-manifolds}, Duke Math. J. \textbf{148}
  (2009), no.~2, 175--209. \MR{2524494 (2010c:53127)}

\bibitem{OzSz:properties}
Peter Ozsv{\'a}th and Zolt{\'a}n Szab{\'o}, \emph{Holomorphic disks and
  three-manifold invariants: properties and applications}, Ann. of Math. (2)
  \textbf{159} (2004), no.~3, 1159--1245. \MR{2113020 (2006b:57017)}

\bibitem{OzSz:lens_surgeries}
\bysame, \emph{On knot {F}loer homology and lens space surgeries}, Topology
  \textbf{44} (2005), no.~6, 1281--1300. \MR{2168576 (2006f:57034)}

\bibitem{OzSz:branched_double_covers}
\bysame, \emph{On the {H}eegaard {F}loer homology of branched double-covers},
  Adv. Math. \textbf{194} (2005), no.~1, 1--33. \MR{2141852 (2006e:57041)}

\bibitem{Petkova:relative_grading}
Ina Petkova, \emph{The decategorification of bordered {H}eegaard {F}loer
  homology}, 2012.

\bibitem{Petkova:absolute_grading}
\bysame, \emph{An absolute {Z}/2 grading on bordered {H}eegaard {F}loer
  homology}, 2014.

\end{thebibliography}
\bibliographystyle{amsplain}

\end{document}